\documentclass[11pt]{amsart}
\usepackage{amsfonts,amscd,amsthm,amsgen,amsmath,amssymb}
\usepackage{epstopdf}
\usepackage{epsfig,color}
\usepackage[all]{xy}
\usepackage[vcentermath]{youngtab}
\usepackage{tikz}
\usetikzlibrary{decorations.markings}
\usepackage[letterpaper, hmargin=1.3in]{geometry}
\usepackage{chngpage}
\usepackage{booktabs}
\newtheorem{theorem}{Theorem}[section]
\newtheorem{lemma}[theorem]{Lemma}
\newtheorem{corollary}[theorem]{Corollary}
\newtheorem{proposition}[theorem]{Proposition}

\theoremstyle{definition}
\newtheorem{definition}[theorem]{Definition}
\newtheorem{example}[theorem]{Example}

\theoremstyle{remark}
\newtheorem{remark}[theorem]{Remark}

\numberwithin{equation}{section}

\begin{document}
\tikzset{->-/.style={decoration={
  markings,
  mark=at position #1 with {\arrow{>}}},postaction={decorate}}}

\tikzset{-<-/.style={decoration={
  markings,
  mark=at position #1 with {\arrow{<}}},postaction={decorate}}}

\title[A foliated Hitchin-Kobayashi correspondence]{A foliated Hitchin-Kobayashi correspondence}

\author{David Baraglia}

\address{School of Mathematical Sciences, The University of Adelaide, Adelaide SA 5005, Australia}
\email{david.baraglia@adelaide.edu.au}

\author{Pedram Hekmati}

\address{Department of Mathematics, The University of Auckland, Auckland, 1010, New Zealand}
\email{p.hekmati@auckland.ac.nz}

\begin{abstract}
We prove an analogue of the Hitchin-Kobayashi correspondence for compact, oriented, taut Riemannian foliated manifolds with transverse Hermitian structure. In particular, our Hitchin-Kobayashi theorem holds on any compact Sasakian manifold. We define the notion of stability for foliated Hermitian vector bundles with transverse holomorphic structure and prove that such bundles admit a basic Hermitian-Einstein connection if and only if they are polystable. Our proof is obtained by adapting the proof by Uhlenbeck and Yau to the foliated setting. We relate the transverse Hermitian-Einstein equations to higher dimensional instanton equations and in particular we look at the relation to higher contact instantons on Sasaki manifolds. For foliations of complex codimension $1$, we obtain a transverse Narasimhan-Seshadri theorem. We also demonstrate that the weak Uhlenbeck compactness theorem fails in general for basic connections on a foliated bundle. This shows that not every result in gauge theory carries over to the foliated setting.
\end{abstract}

\thanks{This work is supported by the Australian Research Council Discovery Projects DP110103745, DP170101054 and the Royal Society of New Zealand Marsden Fund Grant 17-UOA-061.}


\date{\today}


\maketitle



\section{Introduction}

In this paper we prove an analogue of the Hitchin-Kobayashi correspondence for compact, taut Riemannian foliated manifolds with transverse Hermitian structure. The two sides of the correspondence, polystable holomorphic bundles and Hermitian-Einstein connections, are replaced by their foliated analogues. Recall that one motivation for the Hitchin-Kobayashi correspondence comes from studying moduli spaces of the {\em anti-self-dual instanton} equations on a $4$-manifold $X$:
\begin{equation}\label{equ:inst1}
*F_A = -F_A,
\end{equation}
where $F_A$ is the curvature of a unitary connection $A$ on a vector bundle $E \to X$. In general it is difficult to understand the topology of the moduli space of instantons on $X$. However if $X$ is a Hermitian $4$-manifold, one observes that the anti-self-duality equations are equivalent to the equations:
\begin{equation*}
F_A^{0,2} = 0 \quad \quad \Lambda F_A = 0,
\end{equation*}
where $\Lambda$ is the adjoint of the Lefschetz operator. These are (a special case of) the Hermitian-Einstein equations for a unitary connection $A$ on a holomorphic vector bundle. Note that by the first equation $A$ determines a holomorphic structure on $E$. Assume that the Hermitian metric on the $4$-manifold $X$ is Gauduchon, i.e. $\overline{\partial} \partial \omega = 0$. Then the Hitchin-Kobayashi correspondence gives necessary and sufficient conditions for a holomorphic vector bundle $E$ to admit a Hermitian-Einstein connection, namely $E$ must be polystable. Consequently, the moduli space of instantons on $X$ of given rank can be identified with the moduli space of polystable holomorphic vector bundles of degree zero and corresponding rank. The advantage of this is that it allows one to study the moduli space of instantons on $X$ using complex analytic tools.\\

Similarly one can study the Hermitian-Einstein equations in higher dimensions. Suppose $X$ is a Hermitian manifold of real dimension $2n$. We assume the metric on $X$ is Gauduchon, which means that $\overline{\partial} \partial (\omega^{n-1}) = 0$. Let $E$ be a unitary vector bundle on $X$ and $A$ a unitary connection. The Hermitian-Einstein equations with Einstein factor $\gamma_A$ are:
\[
F_A^{0,2} = 0 \quad \quad i \Lambda F_A = \gamma_A Id_E.
\]
For $2n > 4$ these equations have an interpretation as a higher-dimensional analogue of the instanton equation (\ref{equ:inst1}), at least when $E$ has trivial determinant and $\gamma_A = 0$. Namely, they are solutions of the {\em $\Omega$-instanton equations} \cite{cdfn,dt,tian}:
\[
*F_A = -\Omega \wedge F_A, \quad \quad \text{where } \Omega = \frac{ \omega^{n-2} }{(n-2)!}.
\]

Our original motivation for studying a foliated version of the Hitchin-Kobayashi correspondence arose from our study of {\em contact instantons} \cite{bh,kaza}, which are an analogue of the anti-self-dual instanton equations for $5$-dimensional contact manifolds. Let $X$ be a contact $5$-manifold with contact $1$-form $\eta$. The anti-self-dual contact instanton equations are:
\begin{equation}\label{equ:ci1}
*F_A = - \eta \wedge F_A.
\end{equation}
Notice that these are $\Omega$-instantons for $\Omega = \eta$. In \cite{bh}, we studied the moduli space of contact instantons for compact $K$-contact $5$-manifolds. Let $\xi$ denote the Reeb vector field of the contact manifold $X$. Then $\xi$ generates a $1$-dimension foliation of $X$. If $X$ is $K$-contact, then this is a taut Riemannian foliation and the contact instanton equations (\ref{equ:ci1}) can be interpreted as saying that $A$ is a basic connection with respect to the foliation and which is anti-self-dual in the directions transverse to the foliation.\\

Suppose now that the contact $5$-manifold $X$ is Sasakian (the definition is recalled in Section \ref{sec:sasaki}). Recall that Sasakian geometry is an odd-dimensional analogue of K\"ahler geometry. In particular the geometry of $X$ transverse to the Reeb foliation is K\"ahler. In this case the contact instanton equations admit an interpretation as being connections which are Hermitian-Einstein transverse to the foliation. More generally, if one considers Sasaki manifolds of dimension $2n+1$ with $n \ge 2$ one can consider the following version of the $\Omega$-instanton equations:
\begin{equation*}
*F_A = -\Omega \wedge F_A, \quad \quad \text{where } \Omega = \eta \wedge \frac{ \; \; \; \; (d\eta)^{n-2} }{(n-2)!}.
\end{equation*}
As in the $5$-dimensional case, such connections can be interpreted as basic connections which are Hermitian-Einstein transverse to the Reeb foliation.\\

From the above considerations we are lead to consider the following very general setup: let $X$ be a compact oriented manifold with a taut Riemannian foliation which has a transverse Hermitian structure (see Section \ref{sec:tg} for definitions). In Theorem \ref{thm:trgauduchon} we prove that after rescaling by a positive smooth basic function, we can assume the transverse metric is transverse Gauduchon (see Definition \ref{def:trgauduchon}). Since $X$ is transverse Hermitian, we may speak of basic differential forms on $X$ of type $(p,q)$ and we may also define the adjoint $\Lambda$ of the Lefschetz operator, which we regard as acting on basic forms. We may then define transverse Hermitian-Einstein connections as follows:

\begin{definition}
\leavevmode
\begin{itemize}
\item[(i)]{Let $E$ be a foliated Hermitian vector bundle. A basic unitary connection $A$ on $E$ is called {\em transverse Hermitian-Einstein} if its curvature $2$-form $F_A$ is of type $(1,1)$ and satisfies
\[
i \Lambda F_A = \gamma_A id_E,
\]
for some real constant $\gamma_A$, called the {\em Einstein factor} of $A$.}
\item[(ii)]{Let $E$ be a foliated holomorphic vector bundle. A transverse Hermitian metric on $E$ is called a {\em transverse Hermitian-Einstein metric} if the associated Chern connection is transverse Hermitian-Einstein.}
\end{itemize}
\end{definition}

This is the one side of the transverse Hitchin-Kobayashi correspondence. The other side of the correspondence is the foliated analogue of polystable holomorphic vector bundles. Let $E$ be a transverse holomorphic vector bundle on $X$. The degree $\deg(E)$ of a transverse holomorphic vector bundle is defined as follows. Suppose that $E$ admits a transverse Hermitian metric. Note that a transverse Hermitian metric need not exist. Indeed the usual way that one proves the existence of a Hermitian metric on a complex vector bundle is to use a partition of unity. However this fails in the transverse setting since we usually can not find a partition of unity subordinate to a given cover by {\em basic} functions. Thus, we will only define $\deg(E)$ in the case that $E$ admits a transverse Hermitian metric $h$. Let $A$ be the associated Chern connection and $F_A$ its curvature $(1,1)$-form. We define
\begin{equation}\label{equ:defdeg}
\deg(E) = \frac{i}{2\pi} \int_X tr(F_A) \wedge \omega^{n-1} \wedge \chi,
\end{equation}
where $\chi$ is the leafwise volume form. We show that if $X$ is transverse Gauduchon and the foliation is taut, then $\deg(E)$ does not depend on the choice of transverse Hermitian metric.\\

In order to define stability/semistability we need the notion of a transverse coherent subsheaf $\mathcal{F} \subset \mathcal{O}(E)$, where $\mathcal{O}(E)$ is the sheaf of basic holomorphic sections of $E$. Transverse coherent sheaves are introduced in Definition \ref{def:tcs}. To each transverse coherent sheaf $\mathcal{F}$, one can define a determinant line bundle $\det(\mathcal{F})$ in much the same way as  done in the non-foliated setting. We have that $\det(\mathcal{F})$ is a transverse holomorphic line bundle. At this point a complication arises compared to the non-foliated setting, namely, it is by no means clear whether the line bundle $\det(\mathcal{F})$ admits a transverse Hermitian metric and so we can not define the degree of $\det(\mathcal{F})$ simply by using (\ref{equ:defdeg}). To get around this problem, we are lead to consider a foliated version of Hironaka's resolution of singularities, which we carry out in Section \ref{sec:trressing}. This is similar to the approach to stability taken in \cite{lt2}. We use this to define the degree of a transverse coherent subsheaf $\mathcal{F} \subset \mathcal{O}(E)$ of a transverse holomorphic vector bundle $E$ which admits a transverse Hermitian metric, under the further assumption that quotient sheaf $\mathcal{O}(E)/\mathcal{F}$ is torsion-free. This is sufficient to define stability/semistability:

\begin{definition}
Let $E$ be a transverse holomorphic vector bundle which admits a transverse Hermitian metric. We say that $E$ is {\em stable} (resp. {\em semistable}) if for each transverse coherent subsheaf $\mathcal{F}$ of $E$ with $0 < rk(\mathcal{F}) < rk(E)$ and such that the quotient $\mathcal{O}(E)/\mathcal{F}$ is torsion-free, we have
\[
\deg(\mathcal{F})/rk(\mathcal{F}) < \deg(E)/rk(E) \quad ({\rm resp.} \; \; \deg(\mathcal{F})/rk(\mathcal{F}) \le \deg(E)/rk(E) ).
\]
We also say that $E$ is {\em polystable} if $E$ the direct sum of stable bundles of the same slope.
\end{definition}

With these definitions at hand, we may now state our main theorem:
\begin{theorem}[The transverse Hitchin-Kobayashi correspondence]\label{thm:trhk1}
Let $E$ be a transverse holomorphic vector bundle which admits transverse Hermitian metrics. Then $E$ admits a transverse Hermitian metric $h$ which is Hermitian-Einstein if and only if $E$ is polystable. Moreover, if $E$ is simple then $h$ is unique up to constant rescaling.
\end{theorem}

Our proof is based on the Uhlenbeck-Yau method of continuity proof of the Hitchin-Kobayashi correspondence given in \cite{uy} and its exposition in the book \cite{lt}. The overall strategy for proving Theorem \ref{thm:trhk1} is essentially that of Uhlenbeck-Yau. Therefore most of the work involved in the proof is in adapting each step of the proof to the foliated setting. For example, this requires the introduction of Sobolev spaces of {\em basic} sections, establishing embedding, compactness and elliptic regularity results in the basic setting. Working transverse to a foliation means that we are working with transversally elliptic operators which are not genuinely elliptic and so we also need to make use of the theory of such operators \cite{elk}.\\

At this point, the reader may have the impression that essentially any result in gauge theory can be carried over, more or less trivially, to the foliated setting. We wish to emphasise that this is not the case. In fact, as we will demonstrate in Section \ref{sec:appl}, the foliated analogue of the weak Uhlenbeck compactness theorem fails. More specifically, we give an example (Example \ref{ex:uhlenbeckcomp}) of a compact manifold $X$ with a taut Riemannian foliation which is transverse K\"ahler for which one can find sequences of basic connections $A_i$ on a foliated bundle whose curvatures are uniformly bounded (in our example the $A_i$ are flat) such that there is no weakly convergent subsequence modulo basic gauge transformations (in the $L^{p,k}$-norm for any $p,k$).\\

In Section \ref{sec:trns}, we consider the foliated Hitchin-Kobayashi correspondence in the case where the foliation has complex codimension $1$, the foliated analogue of a Riemann surface. A number of simplifications occur here, for example, in the definition of stability/semistability, it is enough to consider transverse holomorphic subbundles. So there is no need to consider transverse coherent sheaves or foliated resolutions of singularities. We refer to this special case of Theorem \ref{thm:trhk1} as the ``transverse Narasimhan-Seshadri theorem":
\begin{theorem}[Transverse Narasimhan-Seshadri theorem]
Let $X$ be a compact oriented, taut, transverse Hermitian foliated of complex codimension $n=1$. Let $E$ be a transverse holomorphic vector bundle which admits transverse Hermitian metrics. Then $E$ admits a transverse hermitian metric such that the Chern connection $A$ satisfies
\[
F_A = - 2\pi i \frac{\mu(E)}{Vol(X)} \omega \otimes Id_E,
\]
if and only if $E$ is polystable.
\end{theorem}
When $E$ has degree $0$, this reduces to the condition that $A$ is a flat connection. In this case, the transverse Narasimhan-Seshadri correspondence can be neatly summarised as follows:

\[
\left\{ \begin{array}{c} \text{Isomorphism classes of} \\ \text{rank } m \text{ unitary} \\ \text{ representations of } \pi_1(X) \end{array} \right\} \leftrightarrow
\left\{ \begin{array}{c} \text{Isomorphism classes of polystable rank } m \\ \text{degree } 0 \text{ transverse holomorphic vector bundles} \\ \text{admitting transverse Hermitian metrics} \end{array} \right\}
\]
\\

In the case of complex codimension $1$ foliations we also prove the analogue of the Harder-Narasimhan filtration:
\begin{theorem}[Transverse Harder-Narasimhan filtration]
Let $X$ be a compact oriented, taut, transverse Hermitian foliated of complex codimension $n=1$. Let $E$ be a transverse holomorphic vector bundle which admits transverse Hermitian metrics. There exists a uniquely determined filtration of $E$
\[
0 = E_0 \subset E_1 \subset \cdots \subset E_k = E
\]
by transverse holomorphic subbundles such that the quotients $F_i = E_i/E_{i-1}$ are semistable and the slopes are strictly increasing:
\[
\mu(F_1) > \mu(F_2) > \cdots > \mu(F_k).
\]
\end{theorem}
To prove the Harder-Narasimhan filtration, one has to show that there exists a transverse holomorphic subbundle $F \subseteq E$ which maximises $\mu(F)$. In the non-foliated setting this is easy to show using the fact that $\mu(F)$ is a rational number with denominator of absolute value at most $rk(E)$. In the foliated setting, the slopes $\mu(F)$ can be any real numbers and it is non-trivial to see that the supremum over all slopes of subbundles of $E$ is attained by a subbundle. To prove this, we make use of the notion of (a foliated analogue of) weakly holomorphic subbundles, which were introduced by Uhlenbeck and Yau in their proof of the Hitchin-Kobayashi correspondence.\\

In Section \ref{sec:sasaki} we return to our original motivation, the case where $X$ is Sasakian. We recall the definition of Sasaki manifolds and observe that since they are transverse K\"ahler, Theorem \ref{thm:trhk1} applies:
\begin{corollary}\label{cor:hesasaki1}
The transverse Hitchin-Kobayashi correspondence holds on any compact Sasaki manifold.
\end{corollary}

Of course, Theorem \ref{thm:trhk1} can be applied to any compact oriented $X$ with taut, transverse K\"ahler foliation. Further examples of such geometries include $3$-Sasaki manifolds and co-K\"ahler manifolds.\\

In Section \ref{sec:instantons}, we re-examine the relation between transverse Hermitian-Einstein connections and $\Omega$-instantons. In general, transverse Hermitian-Einstein connections with trivial determinant correspond to solutions of:
\begin{equation}\label{equ:hechi1}
*F_A = -\Omega \wedge F_A, \quad \quad \text{where } \Omega = \chi \wedge \frac{ \omega^{n-2} }{(n-2)!},
\end{equation}
where we recall that $\chi$ is the leafwise volume form. When $X$ is Sasakian of dimension $2n+1$ (with $n \ge 2$), we have $\chi = \eta$, $\omega = d\eta$ and so (\ref{equ:hechi1}) reduces to the anti-self-dual ``higher contact instanton" equations:
\[
*F_A = - \eta \wedge \frac{ \; \; \; \; (d\eta)^{n-2} }{(n-2)!} \wedge F_A.
\]
Combined with Corollary \ref{cor:hesasaki1}, we have:
\begin{corollary}
Anti-self-dual $SU(r)$ higher contact instantons on a $2n+1$-dimensional Sasakian manifold $X$ (with $n \ge 2)$ correspond to rank $r$ transverse Hermitian-Einstein connections with trivial determinant. If $X$ is compact, then by the transverse Hitchin-Kobayashi correspondence, anti-self-dual $SU(r)$ higher contact instantons on $X$ correspond to rank $r$ polystable transverse holomorphic bundles on $X$ with trivial determinant.
\end{corollary}

The following is a brief summary of each section of the paper. In Section \ref{sec:tg}, we cover the background material needed to study gauge theory transverse to a taut Riemannian foliation. In particular, we review the notions of foliated vector bundles, basic connections and basic differential operators. In Section \ref{sec:sobolev} we introduce Sobolev spaces of basic sections of a vector bundle and prove basic versions of the Sobolev embedding and compactness theorems. In Section \ref{sec:trcomplex}, we study the foliated analogues of various notions in complex geometry. In Section \ref{sec:trherm}, we prove the existence of Gauduchon metric on compact taut transverse Hermitian manifolds. In Section \ref{sec:trressing}, we prove a foliated version of resolution of singularities which we use in Section \ref{sec:stability} to define stability of transverse holomorphic vector bundles. In Section \ref{sec:trhe} we introduce the transverse Hermitian-Einstein equations and prove that transverse Hermitian-Einstein implies polystable. In Section \ref{sec:trhk}, we complete the transverse Hitchin-Kobayashi correspondence by proving the converse. Section \ref{sec:appl} is concerned with some applications and related results. In Section \ref{sec:trns} we consider the case of foliations of complex codimension $1$, which leads to a transverse Narasimhan-Seshadri theorem. We also prove a transverse version of the Harder-Narasimhan filtration and give an example to show the Uhlenbeck compactness fails in general for basic connections on a foliated bundle. In Section \ref{sec:sasaki} we recall the definition of Sasaki manifolds and consider the transverse Hitchin-Kobayashi correspondence for them. Finally, in Section \ref{sec:instantons} we relate the transverse Hermitian-Einstein equations to higher dimensional instanton equations and in particular we look at the relation to higher contact instantons on Sasaki manifolds.

\section{Transverse geometry on Riemannian foliations}\label{sec:tg}

\subsection{Riemannian foliations}\label{sec:riemfol}

Let $X$ be a smooth oriented manifold with a foliation $\mathcal F$ of dimension $m$ and codimension $2n$. We denote by $V=T\mathcal F$ the tangent distribution of the foliation and $H = TX /V$ the normal bundle. We assume that $\mathcal{F}$ is a {\em Riemannian foliation} \cite{mol} which means that $X$ is endowed with a bundle-like metric $g$, so that the foliation is locally identified with a Riemannian submersion. In other words,  $g$ yields an orthogonal splitting $TX = V \oplus H$ and induces
a holonomy invariant Riemannian structure on $H$. Let $vol_T$ and $\chi$ denote  the transverse and leafwise volume forms respectively. These are determined by the metric and the orientations on $H$ and $V$, which are chosen such that  $vol_X=vol_T\wedge \chi$.\\

By a {\em foliated chart} on $X$, we mean a coordinate chart $(U , \varphi)$, where $U \subseteq X$ is an open subset, $\varphi : U \to V \times W \subseteq \mathbb{R}^{2n} \times \mathbb{R}^m$ a diffeomorphism, where $V \subseteq \mathbb{R}^{2n}$, $W \subseteq \mathbb{R}^{m}$ are open subsets and $\varphi(\mathcal{F}|_U)$ is given by the fibres of the projection $V \times W \to V$. When working with foliated charts we will often drop explicit mention of $\varphi$ and simply identify $U$ with $V \times W$.\\

A smooth function $f$ on $X$ is called {\em basic} if $\xi(f) = 0$ for all $\xi \in \Gamma(X,V)$. More generally, a differential form $\alpha \in \Omega^k(X)$ is called  basic if $i_\xi \alpha = 0$ and $i_\xi d \alpha = 0$ for all $\xi \in \Gamma(X,V)$. The meaning of this condition is that in a local foliated chart the form depends only on the transverse variables. We let $\Omega^k_B(X)$ denote the space of basic $k$-forms and $d_B \colon \Omega^k_B(X) \to \Omega^{k+1}_B(X)$ the restriction of the exterior derivative $d$ to basic forms.\\

A foliation is said to be \emph{taut} if $X$ admits a metric such that every leaf of $\mathcal F$ is a minimal submanifold, or equivalently such that the mean curvature form   of the leaves $\kappa$   is trivial. 

\begin{proposition}[Basic Stokes' theorem] 
Let $X$ be a closed oriented manifold with a taut Riemannian foliation of codimension $2n$, then 
\begin{equation}\label{equ:basicstokes}
\int_X d_B \alpha  \wedge \chi = 0
\end{equation}
for all $\alpha \in \Omega_B^{2n-1}(X)$.
\end{proposition}
\begin{proof}  Rummler's formula \cite{rum} states that

$$d\chi = -\kappa \wedge \chi + \phi$$ 
for some $\phi\in \Omega^{m+1}(X)$ satisfying $\iota_{\xi_1}\dots \iota_{\xi_m}\phi=0$ for  any set $\{\xi_j\}$ of $m$ vectors  in $\Gamma(X, V)$.  Therefore $\alpha \wedge \phi=0$ for any basic $(2n-1)$-form $\alpha$ and by tautness it follows that $d_B \alpha \wedge \chi = d( \alpha \wedge \chi )$.\\ 
\end{proof}

A principal $G$-bundle $\pi \colon P \to X$ is said to be a {\em foliated principal bundle} if there is a rank $m$ foliation $\widetilde{\mathcal{F}}$ on $P$ with tangent distribution $T\widetilde{\mathcal{F}} \subseteq TP$ such that $\widetilde{\mathcal{F}}$ is $G$-invariant and $T\widetilde{\mathcal{F}}$ projects isomorphically onto $T\mathcal{F}$ under $\pi_*$. An isomorphism of foliated principal $G$-bundles $(P,\widetilde{\mathcal{F}}),(P',\widetilde{\mathcal{F}}')$ is a principal bundle isomorphism $\phi \colon  P \to P'$ such that $\phi(\widetilde{\mathcal{F}}) = \widetilde{\mathcal{F}}'$.  A local section $s\colon U\to P$ is called basic if $s_* : TU \to TP$ sends $T\mathcal{F}$ to $T\widetilde{\mathcal{F}}$. It is known that every foliated principal bundle admits local basic sections \cite[Proposition 2.7]{mol}. The transition functions $\{ g_{\alpha \beta} \}$ between local basic sections are $G$-valued basic functions. Conversely, if a principal bundle $P$ is defined by a collection $\{ g_{\alpha \beta} \}$ of basic transition functions, this determines a foliated structure on $P$. This is because for each $\alpha$ the obvious lift to $U_\alpha \times G$ of the foliation $\mathcal{F}|_{U_\alpha}$ is preserved by the transition functions. In this way we obtain an equivalence between foliated principal $G$-bundles and equivalence classes of basic \v{C}ech cocycles $\{ g_{\alpha \beta} \}$.\\

A connection $A$ on $P$, thought of as a $\mathfrak{g}$-valued $1$-form on $P$ is called {\em adapted} if $T\widetilde{\mathcal{F}}$ lies in the kernel of $A$. It is called {\em basic} if as a $\mathfrak g$-valued 1-form, $A$ is basic with respect to the foliation $\widetilde{\mathcal{F}}$. While every foliated principal bundle admits an adapted connection by extending $T\widetilde{\mathcal{F}}$ to a horizontal distribution, there is an obstruction to the existence of basic connections. This is ultimately due to the lack of a basic partition of unity subordinate to a given cover and the obstruction is a secondary characteristic class of the foliated bundle \cite{kato, mol}. However, if $A$ is a connection on $P$ for which the curvature $F_A$ satisfies $i_\xi F_A = 0$ for all $\xi \in T\mathcal{F}$, then the horizontal lift of $V$ with respect to $A$ is an integrable distribution and gives $P$ a foliated structure for which $A$ is basic. We note that if $X$ is a Riemannian foliation, then $H$ admits a basic connection, namely we can take the basic Levi-Civita connection associated to the basic metric on $H$.\\

A vector bundle $E$ is {\em foliated} if its associated frame bundle is foliated. In this case we will also say that $E$ has a {\em transverse structure} or that $E$ is a {\em transverse vector bundle}. We denote by $\Gamma_B(X,E)$ the space of smooth basic sections of $E$. In the \v{C}ech language a basic section is given by a collection $\{ s_\alpha \}$ of basic vector-valued functions such that $s_\alpha = g_{\alpha \beta} s_\beta$. A {\em transverse Hermitian metric} $h$ on a foliated complex vector bundle $E$ is a Hermitian metric on $E$ which is basic as a section of $E^* \otimes \overline{E}^*$. Equivalently, $h$ corresponds to a reduction of structure group of the frame bundle from $GL(n,\mathbb C)$ to $U(n)$ as a foliated principal bundle. Such a reduction is however not always possible  and therefore not every foliated complex vector bundle admits a transverse Hermitian metric.\\

Let $E$ be a foliated vector bundle. For each $d \ge 0$, one can define the {\em $d$-th basic jet bundle} $J_B^d(E)$ whose fibres are the $d$-jets of local basic sections of $E$. Clearly $J_B^d(E)$  is a foliated vector bundle in a natural way and there are short exact sequences
\begin{equation}\label{equ:jetbundles}
0 \to Sym^d( H^* ) \otimes E \to J^d_B(E) \to J^{d-1}_B(E) \to 0.
\end{equation}
A basic section of $E$ determines a basic section of $J^d_B(E)$ by prolongation. Suppose that $E$ admits a basic connection $\nabla$. We can use the basic connection to split the sequences (\ref{equ:jetbundles}) and thus non-canonically identify $J^d_B(E)$ with $\bigoplus_{j = 0}^d Sym^j( H^* ) \otimes E$, as foliated vector bundles. Suppose that $F$ is another foliated vector bundle. We define a {\em basic linear differential operator} of order $d$ from $E$ to $F$ to be a basic section of the foliated bundle $Diff^d_B(E,F) = Hom( J^d_B(E) , F)$. Clearly a basic linear differential operator defines a linear map $D : \Gamma_B( X , E) \to \Gamma_B(X , F)$ which in local foliated coordinates is given by a linear differential operator in the transverse coordinates. From the exact sequence (\ref{equ:jetbundles}), there is a natural map $\sigma : Diff^d_B(E,F) \to Hom( Sym^d( H^*) , Hom(E,F) )$ and we define the {\em symbol} of a basic linear differential operator $D$ of order $d$ to be the image $\sigma(D)$ of $D$ under this map. We say that $D$ is {\em transversally elliptic} if $\sigma(D)(\xi , \dots , \xi)$ is invertible for each $0 \neq \xi \in H^*$.

\begin{remark}\label{rem:extend}
Let $E,F$ be foliated vector bundles and $D \in \Gamma_B( X , Diff^d_B(E,F))$ a basic linear differential operator of order $d$. If $E$ admits a basic connection, then $D$ can be non-canonically extended to a linear differential operator in the ordinary sense. To see this, let $J^d(E)$ be the usual $d$-th jet bundle of $E$ and $Diff^d(E,F) = Hom( J^d(E) , F)$. We use the connection to obtain splittings $J^d_B(E) \cong \bigoplus_{j = 0}^d Sym^j( H^* ) \otimes E$, $J^d(E) \cong \bigoplus_{j = 0}^d Sym^j( T^*X ) \otimes E$ and we use the Riemannian metric $g$ to obtain a non-canonical splitting $T^*X = H^* \oplus V^*$ and thus an inclusion $H^* \to T^*X$. In this way, we obtain a non-canonical homomorphism $Diff^d_B(E,F) \to Diff^d(E,F)$.
\end{remark}

Let $E,F$ be foliated vector bundles and $D : \Gamma_B(X,E) \to \Gamma_B(X,F)$ a basic linear differential operator. Assume that $E,F$ admit transverse Hermitian structures, in other words they are vector bundles associated to foliated principal bundles with structure group a unitary group. Then we may construct the formal adjoint operator $D^* : \Gamma_B(X , F) \to \Gamma_B(X , E)$ \cite{elk} which has the property
\[
\langle Ds , t \rangle = \langle s , D^* t \rangle
\]
for all $s \in \Gamma_B(X,E)$, $t \in \Gamma_B(X , F)$, where $\langle \; , \; \rangle$ is the usual $L^2$-pairing of sections of a Hermitian vector bundle. Let us give a construction of $D^*$ under the assumptions that the foliation is taut and that $E$ admits a basic connection $\nabla$. Whenever we take the adjoint of a basic differential operator in this paper, these assumptions will hold. As in Remark \ref{rem:extend} we can use $\nabla$ to extend $D$ to an ordinary linear differential operator, which has the form
\[
D(s) = \sum_{j = 0}^d \sum_{|I| = j} a_{(j)}(e^{i_1} , e^{i_2} , \dots , e^{i_j} ) \nabla_{e_{i_1}} \nabla_{e_{i_2}} \cdots \nabla_{e_{i_j}} s
\]
where $e_1 , \dots , e_{2n}$ is a local frame of basic sections of $H$, $e^1, \dots , e^{2n}$ the dual coframe and $a_{(j)}$ is a basic section of $Sym^j( H ) \otimes Hom(E,F)$. We claim that $D^*$ is given by the usual formula for the adjoint:
\[
D^*(s) = \sum_{j = 0}^d \sum_{|I| = j}  \nabla^*_{e_{i_j}} \cdots \nabla^*_{e_{i_2}} \nabla^*_{e_{i_1}} (a_{(j)}^*(e^{i_1} , e^{i_2} , \dots , e^{i_j} )s ),
\]
where $\nabla^*_{e_i} s = -\nabla_{e_i} s - div(e_i)s$ and $div(e_i)$ is the divergence of $e_i$. To see this, note that by the assumption of tautness, the divergence of $Y \in \Gamma(X , H)$ is determined by $div(Y) vol_T = d_B( i_Y vol_T)$, which shows that if $Y$ is a basic section of $H$, then $div(Y)$ is a basic function. Therefore the formal adjoint $D^*$ is again a basic differential operator.\\

From the theory of transversally elliptic operators, we have:
\begin{theorem}[\cite{elk}]
Let $X$ be a compact Riemannian foliated manifold, $E,F$ foliated complex vector bundles admitting transverse Hermitian metrics and $D : \Gamma_B(X , E) \to \Gamma_B(X , F)$ a basic linear differential operator of order $d$. Then $D$ is Fredholm, that is, $Ker(D)$ and $Ker(D^*)$ are finite-dimensional.
\end{theorem}

Let $E$ be a foliated vector bundle. Denote by $\Omega^k(X,E)$ the space of $k$-form valued sections of $E$ and by $\Omega^k_B(X , E)$ the space of basic $k$-form valued basic sections of $E$. The latter is given by $\Omega^k_B(X,E) = \Gamma_B(X , E \otimes \wedge^k H^* )$. Suppose that $\nabla$ is a basic connection on $E$. We have that $\alpha \in \Omega^k(X , E)$ is a basic section if and only if $i_\xi \alpha = 0$ and $i_\xi d_\nabla \alpha = 0$ for all $\xi \in \Gamma(X,V)$. It follows that the restriction of $d_\nabla : \Omega^k( X , E ) \to \Omega^{k+1}(X,E)$ to basic sections induces a basic first order differential operator $d_\nabla : \Omega^k_B(X , E) \to \Omega^{k+1}_B(X , E)$.

\subsection{Basic Sobolev spaces and elliptic regularity}\label{sec:sobolev}
In this section, we assume $X$ is a compact, oriented, Riemannian foliation of codimension $2n$. 
\begin{definition} Let  $E$ be a foliated vector bundle on $X$ equipped with a transverse Hermitian structure and a compatible basic connection $\nabla$.  For $k$ a non-negative integer and $p\in [1,\infty)$, the {\em basic Sobolev space} $L^{p,k}_B(E)$ is defined as the norm closure in $L^{p,k}(E)$ of the space of smooth basic sections $\Gamma_B(X,E)$ under the Sobolev norm 
$$\|s\|_{p,k}=\left(\sum_{j=0}^k||\nabla^js||^p_{L^p}\right)^{\frac 1p}, \quad s\in \Gamma(X,E).$$
Similarly, for a non-negative integer $k$ we define $\mathcal{C}^k_B(E)$ to be the subspace of $\mathcal{C}^k(E)$ consisting of basic sections. It is easy to see that the limit of a $\mathcal{C}^k$-convergent sequence of basic sections of $E$ is again basic, hence $\mathcal{C}^k_B(E)$ is a closed subspace of $\mathcal{C}^k(E)$.
\end{definition}

\begin{remark}
Note that to define $\mathcal{C}^0_B(E)$, we need to say what is a continuous basic section of $E$. We say that a continuous section of $E$ is basic if in each local basic trivialisation of $E$, the section is given by a basic function, where a continuous function $f$ defined on an open subset $A \subseteq X$ is called basic if for each $x \in A$, there exists a local foliated chart $U = V \times W \subseteq A$ containing $x$ for which $f|_U$ is the pullback of a continuous function on $V$ under the projection $V \times W \to V$. For differentiable sections, this clearly agrees with our previous notion of basic sections. It is also clear that this notion of basic sections is closed under $\mathcal{C}^0$-limits.
\end{remark}

\begin{theorem}[Basic Sobolev embedding and compactness] 
\leavevmode
\begin{itemize}
\item[(i)]{
For all integers $k,l$ such that $k \geqslant l\geqslant 0$ and for all $p,q \in [1,\infty)$ such that $k - \frac{2n}{p}\geqslant l-\frac{2n}{q}$, there is a continuous inclusion 
$$ L_B^{p,k}(E)\hookrightarrow L_B^{q,l}(E). $$
Moreover, if $k > l$ and $k - \frac{2n}{p} > l-\frac{2n}{q}$ then the inclusion is compact.
}
\item[(ii)]{
For all integers $k,l$ with $k > l \geqslant 0$ and for all $p \in [1,\infty)$ such that $k - \frac{2n}{p} > l$, there is a continuous inclusion
\[
L_B^{p,k}(E) \hookrightarrow \mathcal{C}_B^l(E).
\]
Moreover, the inclusion is compact.
}
\end{itemize}
\end{theorem}

\begin{proof}
The argument is essentially that of \cite[Theorem 9, 10]{klw} which we now sketch. Consider for instance the case $L_B^{p,k}(E)\hookrightarrow L_B^{q,l}(E)$. Let $f \in L_B^{p,k}(E)$. Then $f$ is the $L^{p,k}$-limit of a sequence $f_i$ of basic smooth functions. Consider a local foliated chart $U = V \times W$ over which $E$ is trivialised as a foliated bundle. We assume that the foliated chart is chosen so that the closure $\overline{U}$ is contained in a slightly larger foliated chart $\tilde{U} = \tilde{V} \times \tilde{W}$ and that the local trivialisation of $E$ extends to $\tilde{U}$. We also assume that $\tilde{V} \subset \mathbb{R}^{2n}$, $\tilde{W} \subset \mathbb{R}^{m}$ are bounded. Then $f_i|_U$ is given by a smooth vector-valued function on $V$. Taking the $L^{p,k}$-limit, we see that $f|_U$ is given by an $L^{p,k}$ function on $V$. Under the stated assumptions, the Sobolev embedding theorem for $V$ gives a continuous injection $L^{p,k}(V) \to L^{q,l}(V)$ (see for instance \cite{adams}). By considering a collection of such foliated charts covering $X$, we see that the sequence $\{ f_i \}$ converges in the $L^{q,l}$-norm and therefore $f \in L^{q,l}_B(E)$ by the definition of $L^{q,l}_B(E)$. The Sobolev embedding theorem also implies an estimate $|| f_i ||_{q,l} \le C || f_i ||_{p,k}$, for a constant $C$ which does not depend on $f$ or the $f_i$. Hence $|| f ||_{q,l} \le C || f ||_{p,k}$, so that the inclusion $L_B^{p,k}(E)\hookrightarrow L_B^{q,l}(E)$ is continuous.\\

We now argue that if $k > l$ and $k - \frac{2n}{p} > l-\frac{2n}{q}$, the inclusion is compact. Let $f_i$ be a bounded sequence in $L^{p,k}_B(E)$. Then in each foliated chart $U = V \times W$ of the type previously described, we can apply Sobolev compactness of the inclusion $L^{p,k}(V) \to L^{q,l}(V)$ (see \cite{adams}) to deduce that there is a subsequence of $\{ f_i \}$ which converges in $L^{q,l}( V \times W)$. Since we can cover $X$ by finitely many such charts, we can find a subsequence of $\{ f_{i_j} \}$ which converges in $L^{q,l}(E)$. But since each $f_i$ belongs to the closed subspace $L^{q,l}_B(E)$, it follows that the convegent subsequence $\{ f_{i_j} \}$ converges to an element of $L^{q,l}_B(E)$, which proves compactness. The case of $L_B^{p,k}(E) \hookrightarrow \mathcal{C}_B^l(E)$ is proved similarly.
\end{proof}

\begin{theorem}[Basic Sobolev multiplication]
Let $E,F$ be foliated vector bundles equipped with transverse Hermitian metrics and basic unitary connections. Let $k,k',l$ be integers and $p,p',q \in [1,\infty)$ be such that $\left( k - \frac{2n}{p} \right) + \left( k' - \frac{2n}{p'} \right) > \left( l - \frac{2n}{q} \right)$, $k,k' \ge l$, $(k-l)p < 2n$ and $(k'-l)p' < 2n$.
Then multiplication of smooth basic sections extends to a continuous map
\[
L^{p,k}_B(E) \times L^{p',k'}_B(F) \to L^{q , l}_B( E \otimes F).
\]
\end{theorem}
\begin{proof}
First, let us reduce this result to the case where $k = k' = l$. If $k > l$, then the assumption $(k-l)p < 2n$ ensures that there exists $\tilde{p} \in [1,\infty)$ such that $k - \frac{2n}{p} = l - \frac{2n}{\tilde{p}}$. Then we can use the basic Sobolev embedding theorem to replace $L^{p,k}_B(E)$ with $L^{\tilde{p} , l}_B(E)$. A similar argument applies if $k' > l$.\\

We now assume $k = k' = l$. In this case, $k,p,p',q$ satisfy $\frac{k}{2n} + \frac{1}{q} < \frac{1}{p} + \frac{1}{p'}$. If $k = 0$, then the fact that multiplication extends to a continuous map $L^{p}_B(E) \times L^{p'}_B(F) \to L^{q}_B( E \otimes F)$ follows from the H\"older inequality. Now we proceed by induction on $k$. The rest of the proof is exactly the same as the proof of \cite[Lemma B.3]{weh}, except using the basic Sobolev embedding theorem in place of the ordinary one.
\end{proof}

We also need to use a transverse analogue of elliptic regularity. The following version suffices for our purposes:
\begin{lemma}[Basic elliptic regularity]\label{lem:trreg}
Let $E$ be a transverse vector bundle with admits a transverse Hermitian metric and a basic connection $\nabla$. Let $L : \Gamma_B(X,E) \to \Gamma_B(X,E)$ be a second order transverse elliptic differential operator from $E$ to itself. Suppose that $k \ge 2$ and $p \in [1,\infty)$. Let $s \in L^{p,k}(E) \cap L^{p,1}_B(E)$ be such that $L(s) \in L^{p,k-1}(E)$. Then we have that $s \in L^{p,k+1}(E) \cap L^{p,1}_B(E)$.
\end{lemma}
\begin{proof}
By Remark \ref{rem:extend}, we can (non-canonically) extend $L$ to a second order differential operator $L : \Gamma(X , E) \to \Gamma(X,E)$. Let $\nabla_V : \Gamma( X , E) \to \Gamma(X , E \otimes V^*)$ denote the composition of $\nabla : \Gamma(X , E) \to \Gamma(X , E \otimes T^*X)$ with the orthogonal projection $T^*X \to V^*$. Let $\nabla_V^* : \Gamma(X , E\otimes V^*) \to \Gamma(X , E)$ be the formal adjoint of $\nabla_V$. Define $D : \Gamma(X , E) \to \Gamma(X , E)$ to be the second order differential operator
\[
D(a) = \nabla_V^* \nabla_V a + L(a).
\]
We have that $D$ is elliptic because $L$ is transverse elliptic. Now let $s \in L^{p,k}(E) \cap L^{p,1}_B(E)$ and suppose that $L(s) \in L^{p,k-1}(E)$. Note that $s \in L^{p,1}_B(E)$ implies that $\nabla_V s = 0$. This is because $s$ is the $L^{p,1}$-limit of a sequence $s_i$ of basic smooth sections. Since the $s_i$ are basic and smooth, they satisfy $\nabla_V s_i = 0$, hence also $\nabla_V s = 0$. It follows that $D(s) = L(s) \in L^{p,k-1}(E)$. By the usual elliptic regularity of linear elliptic differential operators, we have $s \in L^{p,k+1}(E)$. 
\end{proof}
\begin{remark}
Note that for $k \ge 1$, we have $L^{p,k}_B(E) \subseteq L^{p,k}(E) \cap L^{p,1}_B(E)$. We do not know whether this inclusion is an equality. Fortunately, the above lemma suffices for the results of this paper.
\end{remark}

\section{Transverse complex geometry}\label{sec:trcomplex}

\subsection{Transverse Hermitian structures and transverse Gauduchon metrics}\label{sec:trherm}

\begin{definition}
Let $X$ be a foliated manifold such that the foliation has codimension $2n$. Let $V = T\mathcal{F}$ be the tangent distribution to the foliation and $H = TX/V$ the normal bundle. A {\em transverse almost complex structure} is an endomorphism $I : H \to H$ such that $I^2 = -Id$. We say that the almost complex structure is {\em integrable} and that $X$ has a {\em transverse complex structure} if $X$ can be covered by foliated charts $U_\alpha = V_\alpha \times W_\alpha$, such that each $V_\alpha$ is an open subset of $\mathbb{C}^{n}$ and such that $I|_{U_\alpha}$ agrees with the natural complex structure on $V_\alpha$ obtained from the inclusion $V_\alpha \subseteq \mathbb{C}^n$.
\end{definition}

\begin{definition}
A {\em transverse Hermitian structure} on a Riemannian foliated manifold $X$ is a pair $(g , I)$ consisting of a bundle-like Riemannian metric $g$ and a transverse integrable complex structure $I$ such that $g$ and $I$ are compatible in the usual sense: $g(IX,IY) = g(X,Y)$ for all $X,Y \in H$. In this case we define the associated Hermitian $2$-form $\omega \in \Gamma( X , \wedge^2 H^*)$ by $\omega( X , Y ) = g(IX , Y)$. Pulling back by the natural projection $TX \to H$, we identify $\omega$ with a $2$-form on $X$. Using that fact that $g$ is bundle-like and $I$ is integrable, one sees that $\omega$ is a real basic $2$-form, $\omega \in \Omega^2_B(X)$.
\end{definition}

Throughout this section, unless stated otherwise we assume $X$ is a compact oriented, taut, transverse Hermitian foliation of complex codimension $n$. We may introduce transverse analogues of all the usual notions in Hermitian geometry. For instance we have the Lefschetz operator:
\[
L : \Omega^j_B(X) \to \Omega^{j+2}_B(X), \quad L(\alpha) = \omega \wedge \alpha
\]
and the contraction operator, the adjoint of the Lefschetz operator:
\[
\Lambda : \Omega^j_B(X) \to \Omega^{j-2}_B(X), \quad \Lambda(\alpha) = L^*(\alpha).
\]
The transverse complex structure $I$ allows us to speak of basic differential forms of type $(p,q)$. We denote the space of such forms as $\Omega^{p,q}_B(X,\mathbb{C})$. The exterior derivative $d$, restricted to basic differential forms can be decomposed as $d = \partial + \overline{\partial}$, where $\partial : \Omega^{p,q}_B(X,\mathbb{C}) \to \Omega^{p+1,q}_B(X,\mathbb{C})$ and $\overline{\partial} : \Omega^{p,q}_B(X,\mathbb{C}) \to \Omega^{p,q+1}_B(X,\mathbb{C})$. Integrability of $I$ ensures that $\partial^2 = \overline{\partial}^2 = 0$. If $E$ is a foliated complex vector bundle we can also define $\Omega^{p,q}_B(X,E)$, the space of basic $(p,q)$-form valued sections of $E$.

\begin{definition}
Let $E$ be a foliated complex vector bundle. A basic $\overline{\partial}$-connection on $E$ is a first order basic differential operator $\overline{\partial}_E : \Omega^0_B(X,E) \to \Omega^{0,1}_B(X,E)$ satisfying $\overline{\partial}_E(fs) = \overline{\partial}(f) s + f \overline{\partial}_E(s)$ for all local basic sections $s$ and local basic functions $f$. We extend $\overline{\partial}_E$ to a basic differential operator $\overline{\partial}_E : \Omega^{p,q}_B(X,E) \to \Omega^{p,q+1}_B(X,E)$ in the usual way. We say that $\overline{\partial}_E$ is {\em integrable} if $\overline{\partial}_E^2 = 0$. A transverse holomorphic structure on $E$ is by definition an integrable transverse $\overline{\partial}$-connection. A basic section of $E$ is said to be {\em holomorphic} if $\overline{\partial}_E s = 0$.
\end{definition}
Let $(E , \overline{\partial}_E)$ be a transverse holomorphic vector bundle. The integrability condition $\overline{\partial}_E^2 = 0$ implies that locally $E$ admits a local frame of basic holomorphic sections. With respect to such frames, the transition functions $\{ g_{\alpha \beta } \}$ are basic holomorphic $GL(r , \mathbb{C})$-valued functions (where $r$ is the rank of $E$). Conversely, given a cocycle $\{ g_{\alpha \beta} \}$ of $GL(r,\mathbb{C})$-valued basic holomorphic functions, the associated vector bundle $E$ has a natural transverse holomorphic structure.\\

Many of the constructions one can do with holomorphic vector bundles have counterparts in the transverse setting. For instance, suppose that $E$ is a transverse vector bundle with transverse Hermitian metric and basic unitary connection $\nabla_E$. Then we can decompose $\nabla_E$ into its $(1,0)$ and $(0,1)$-parts $\nabla_E = \partial_E + \overline{\partial}_E$, when acting on basic sections. Then $\overline{\partial}_E$ is a basic $\overline{\partial}$-operator. Moreover, $\overline{\partial}_E$ is integrable if and only if the curvature of $\nabla_E$ has type $(1,1)$. In the other direction, if $(E , \overline{\partial}_E)$ is a transverse holomorphic vector bundle which admits a transverse Hermitian metric, then we can complete $\overline{\partial}_E$ to a uniquely determined basic unitary connection $\nabla_E = \partial_E + \overline{\partial}_E$, which we call the {\em Chern connection} associated to $\overline{\partial}_E$. Lastly, if $E$ is a transverse holomorphic vector bundle which admits a transverse Hermitian metric, then we can take adjoints of the operators $d_E , \partial_E , \overline{\partial}_E$ and form their associated Laplacians $\Delta_{d_E} = d^*_E d_E + d_E d^*_E$, etc.

\begin{definition}
Let $E$ be a transverse holomorphic vector bundle equipped with a transverse Hermitian metric $h$. The {\em $P$-operator} associated to $(E, h)$, denoted $P_E$ is given by
\[
P_E : \Omega^0_B(E) \to \Omega^0_B(E), \quad P_E = i \Lambda \overline{\partial}_E \partial_E.
\]
When $E$ is the trivial line bundle equipped with the standard Hermitian metric we write $P$ instead of $P_E$.
\end{definition}

\begin{lemma}\label{lem:ind0}
Let $E$ be a transverse holomorphic vector bundle equipped with a transverse Hermitian metric $h$. Then $P_E , P^*_E$ are transverse elliptic operators of index zero, i.e. $dim(Ker(P_E)) = dim( Ker(P_E^*) )$.
\end{lemma}
\begin{proof}
A direct computation shows that the symbol of $P_E$ coincides with the symbol of $\Delta_{\partial_E}$. It follows that $P_E$ is transverse elliptic and that $P_E$ and $\Delta_{\partial_E}$ differ by a first order operator. Now consider $P_E, \Delta_{\partial_E}$ as Fredholm operators $L^{2,2}_B(E) \to L^2_B(E)$. The difference is a first order differential operator, which by basic Sobolev compactness is compact as an operator from $L^{2,2}_B(E)$ to $L^2_B(E)$. The Fredholm operators $P_E$ and $\Delta_{\partial_E}$ differ by a compact operator, so they have the same index. But $\Delta_{\partial_E}$ is self-adjoint, so it has index zero and hence $P_E$ has index zero as well.
\end{proof}

\begin{lemma}[\cite{lt} Lemma 7.2.4]\label{lem:pstar}
Let $E$ be a transverse holomorphic vector bundle equipped with a transverse Hermitian metric $h$. The adjoint $P_E^*$ of the $P$-operator $P_E$ is given by
\[
P_E^* = \frac{i}{(n-1)!} *_B \overline{\partial}_E \partial_E L^{n-1},
\]
where $*_B$ denotes the basic Hodge star $*_B : \wedge^j H^* \to \wedge^{2n-j} H^*$.
\end{lemma}
\begin{proof}
Bearing in mind the discussion in Section \ref{sec:riemfol} on computing the adjoint of a basic differential operator in the case of a taut Riemannian foliation, the computation of the adjoint $P_E^*$ proceeds exactly as in the non-foliated setting \cite[Lemma 7.2.4]{lt}.
\end{proof}

\begin{lemma}\label{lem:kerimp}
For the trivial line bundle, we have $Ker(P) = \mathbb{C}$ and $Im( P|_{\mathcal{C}^\infty_B(X, \mathbb{R})} )$ contains no basic functions of constant sign other than the zero function.
\end{lemma}
\begin{proof}
Both statements follow easily from the maximum principle, as shown in \cite[Lemma 7.2.7]{lt}.
\end{proof}

\begin{corollary}\label{cor:ppstar}
\leavevmode
\begin{itemize}
\item[(i)]{We have $dim(Ker(P^*)) = 1$ and every function $f \in Ker( P^*|_{\mathcal{C}^\infty_B(X,\mathbb{R})})$ has constant sign.}
\item[(ii)]{We have a direct sum decomposition $\mathcal{C}^\infty_B(X,\mathbb{R}) = Im(P|_{\mathcal{C}^\infty_B(X,\mathbb{R})}) \oplus \mathbb{R}$, where the second summand denotes constant real valued functions.}
\end{itemize}
\end{corollary}
\begin{proof}
(i) From Lemmas \ref{lem:ind0} and \ref{lem:kerimp}, we have $dim(Ker(P^*)) = dim(Ker(P)) = 1$. Now suppose that $f \in Ker( P^*|_{\mathcal{C}^\infty_B(X,\mathbb{R})} )$ is positive at some points and negative at some other points. So there exists an interval $[a_+ , b_+] \subset \mathbb{R}$ with $b_+ > a_+ > 0$ for which $f^{-1}( [a_+ , b_+ ] )$ has positive measure. Let $\varphi_+ : \mathbb{R} \to \mathbb{R}$ be a smooth non-negative function with support on the positive real axis and with $\varphi|_{[a_+ , b_+]} = 1$. Let $g_+ = \varphi_+ \circ f : X \to \mathbb{R}$. Then $g_+$ is a basic smooth function on $X$, $g_+ \ge 0$ everywhere and $I_+ = \int_X f \cdot g_+ dvol_X > 0$. Similarly, since $f$ is negative somewhere, we can find a smooth non-negative function $\varphi_- : \mathbb{R} \to \mathbb{R}$ with support on the negative real axis such that $g_- = \varphi_- \circ f : X \to \mathbb{R}$ is a basic smooth function on $X$, $g_- \ge 0$ and $I_- = \int_X f \cdot g_- dvol_X < 0$. Now setting $g = I_+ g_- -I_- g_+$ we have that $g$ is a basic smooth function, $g \ge 0$ everywhere and $\int_X f \cdot g dvol_X = 0$. Since $Ker(P^*|_{\mathcal{C}^\infty_B(X,\mathbb{R})}  )$ is $1$-dimensional, this shows that $g \in Ker(P^*)^\perp = Im(P|_{\mathcal{C}^\infty_B(X,\mathbb{R})} )$. But $g$ has constant sign, so this contradicts Lemma \ref{lem:kerimp}. Lemma \ref{lem:kerimp} also implies that $Im(P|_{\mathcal{C}^\infty_B(X,\mathbb{R})} ) \cap \mathbb{R} = \{ 0 \}$ and this implies (ii), since $dim( Coker(P)) = 1$.
\end{proof}

\begin{definition}\label{def:trgauduchon}
Let $X$ be compact oriented and transverse Hermitian foliated. Let $g$ be the transverse Hermitian metric on $X$. We say $g$ is {\em (transverse) Gauduchon} if $\partial \overline{\partial} (\omega^{n-1}) = 0$.
\end{definition}

\begin{theorem}\label{thm:trgauduchon}
Let $X$ be a compact oriented, taut, transverse Hermitian foliated manifold of complex codimension $n$ and let $g$ be the transverse Hermitian metric on $X$. Then $g$ can be conformally rescaled by a basic positive real valued smooth function such that the rescaled transverse metric $g_0$ is Gauduchon. If $X$ connected and $n \ge 2$, then $g_0$ is the unique transverse Gauduchon metric within its conformal class up to constant rescaling.
\end{theorem}
\begin{proof}
This proof adapts \cite[Theorem 1.2.4]{lt} to the foliated setting. If $n=1$ there is nothing to show, so assume $n \ge 2$. Consider the following second order transverse elliptic differential operator
\[
Q : \mathcal{C}^\infty_B(X) \to \mathcal{C}^\infty_B(X), \quad Q(\varphi) = i *_B \partial \overline{\partial}( \omega^{n-1} \varphi ).
\]
If we can find a smooth basic function $\varphi$ satisfying $Q(\varphi) = 0$ and which is everywhere positive, then $g_0 = \varphi^{\frac{1}{n-1}}g$ will be a transverse Gauduchon metric. By Lemma \ref{lem:pstar}, we have $Q = -(n-1)! P^*$. By Corollary \ref{cor:ppstar}, $Ker(Q)$ is $1$-dimensional, so if $g_0$ exists then it is unique up to scale. Let $\varphi_0$ span $Ker(Q)$. By Corollary \ref{cor:ppstar} we may assume $\varphi_0 \ge 0$. It remains only to show that $\varphi_0$ is non-vanishing. This follows by applying the maximum principle in exactly the same manner as in the proof of \cite[Theorem 1.2.4]{lt}.
\end{proof}

\begin{definition}
Let $E$ be a transverse holomorphic vector bundle equipped with a transverse Hermitian metric $h$. Let $\nabla_E$ be the associated Chern connection. The {\em mean curvature} of $E$, denoted by $K_E \in \Omega^0_B( End(E) )$ is defined as $K_E = i \Lambda F_E$, where $F_E$ is the curvature of $\nabla_E$.
\end{definition}

From the above definition it follows that
\[
i n F_E \wedge \omega^{n-1} = K_E \omega^n.
\]

\begin{lemma}\label{lem:deltappstar}
Suppose that $X$ is a compact oriented, taut, transverse Hermitian foliation of complex codimension $n$ with transverse Gauduchon metric. Let $E$ be a transverse holomorphic vector bundle equipped with a transverse Hermitian metric $h$. Then on $\Omega^0(E)$ we have
\[
\Delta_E(a) = P_E(a) + P^*_E(a) - K_E(a).
\]
\end{lemma}
\begin{proof}
This is a local computation, so it is essentially the same as the non-foliated setting, see \cite[Lemma 7.2.5]{lt}.
\end{proof}

\begin{lemma}
Suppose that $X$ is a compact oriented, taut, transverse Hermitian foliation of complex codimension $n$ with transverse Gauduchon metric. Let $E$ be a transverse holomorphic vector bundle equipped with a transverse Hermitian metric $h$ and suppose that $K_E = 0$. Then (as operators on $\Omega^0_B(E)$) we have:
\[
Ker(P_E) = Ker(P^*_E) = Ker(\Delta_E) = Ker(d_E).
\]
\end{lemma}
\begin{proof}
Clearly $Ker(d_E) = Ker(\Delta_E)$. If $d_E(a) = 0$, then $\partial_E(a) =0$ and hence $P_E(a) = 0$, so $Ker(d_E) \subseteq Ker(P_E)$. On the other hand, if $K_E = 0$, then by Lemma \ref{lem:deltappstar}, we have
\[
\langle \Delta_E a , a \rangle = \langle P_E(a) , a \rangle + \langle a , P_E(a) \rangle = \langle P_E^*(a) , a \rangle +  \langle a , P_E^*(a) \rangle,
\]
from which it follows that $Ker(P_E) \subseteq Ker(\Delta_E) = Ker(d_E)$ and $Ker(P_E^*) \subseteq Ker(\Delta_E) = Ker(d_E)$. It remains only to show that $Ker(d_E) \subseteq Ker(P_E^*)$. But this follows easily from Lemma \ref{lem:pstar} and the fact that $X$ is Gauduchon.
\end{proof}

\subsection{Transverse resolution of singularities}\label{sec:trressing}

\begin{definition}
Let $X$ be a foliated manifold with transverse complex structure. A subset $S \subseteq X$ is called a {\em transverse analytic subvariety of $X$} if for every $s \in S$, there exists a foliated chart $U = V \times W$ containing $s$ for which $S \cap U$ is the common zero locus of a finite number of basic holomorphic functions on $U$. Since $S$ is locally given in a foliated coordinate chart $U = V \times W$ as the pre-image under $V \times W \to V$ of an analytic subvariety $S_V \subseteq V$ in the ordinary sense, we can for each $s \in S$ define the {\em codimension of $S$ at $s$} to be the codimension of the corresponding point in $S_V$ (with codimension taken with respect to $V$). Clearly this does not depend on the choice of foliated chart. We then define the {\em codimension of $S$} to be the infimum over all $s \in S$ of the codimension of $S$ at $s$.\\

The collection of all transverse analytic subsets of $X$ gives a topology on $X$ (which is certainly not Hausdorff as it does not separate points which lie in the same leaf of the foliation). Using this topology, we may speak of {\em irreducible transverse analytic subsets}.
\end{definition}

Any local properties in complex analytic geometry can easily be extended to the setting of transverse analytic subvarieties. Thus for example we may speak of singular or non-singular transverse analytic subvarieties of $X$ and if $S \subseteq X$ is a transverse analytic subvariety we may speak of the singular locus $S_{\rm sing} \subset S$, which is again a transverse analytic subvariety of $X$.

\begin{theorem}
Let $X$ be a foliated manifold with transverse complex structure and let $\iota : Y \to X$ be a transverse analytic subvariety. There exists a foliated manifold $\widetilde Y$ with transverse complex structure and a proper map $q : \widetilde{Y} \to Y$ such that
\begin{itemize}
\item[(i)]{The composition $\iota \circ q : \widetilde{Y} \to X$ is a smooth map.}
\item[(ii)]{$\iota \circ q$ is transverse holomorphic in the sense that there exists covers of $\tilde{Y}$ and $X$ by foliated coordinate charts of the form $U_1 = V_1 \times W_1 \subseteq \widetilde{Y}$ and $U_2 = V_2 \times W_2 \subseteq X$ on which $\iota \circ q$ is given by
\[
V_1 \times W_1 \ni ( v , w ) \mapsto ( q_1(v) , q_2(v,w) ) \in V_2 \times W_2,
\]
where $q_1$ is holomorphic.}
\item[(iii)]{$q : \widetilde{Y} \to Y$ is an isomorphism of foliated manifolds with transverse holomorphic structure over the non-singular part of $Y$.}
\end{itemize}
\end{theorem}
\begin{proof}
This is basically a consequence of the existence of a functorial resolution of singularities for analytic varieties \cite{wlod}. Choose an open cover of $X$ by foliated charts $U_\alpha = V_\alpha \times W_\alpha \buildrel \varphi_\alpha \over \longrightarrow X$. The charts can be chosen so that the overlaps $U_{\alpha \beta} = \varphi_\alpha^{-1}( U_\beta ) \subseteq U_\alpha$ have the form $U_{\alpha \beta} = V_{\alpha \beta} \times W_{\alpha \beta}$ with $V_{\alpha \beta} \subseteq V_\alpha$, $W_{\alpha \beta} \subseteq W_\alpha$ and the transition maps
\begin{equation*}\xymatrix{
V_{\beta \alpha} \times W_{\beta \alpha} = U_{\beta \alpha} = \varphi^{-1}_\beta( U_\alpha ) \ar[rr]^-{\varphi_{\alpha \beta} = \varphi_\alpha^{-1} \circ \varphi_\beta} & & \varphi_\alpha^{-1}(U_\beta) = U_{\alpha \beta} = V_{\alpha \beta} \times W_{\alpha \beta}
}
\end{equation*}
have the form
\[
\varphi_{\alpha \beta}( u , v ) = ( f_{\alpha \beta}(u) , g_{\alpha \beta}(u,v) )
\]
for some $f_{\alpha \beta} : V_{\beta \alpha} \to V_{\alpha \beta}$ and some $g_{\alpha \beta} : V_{\beta \alpha} \times W_{\beta \alpha} \to W_{\alpha \beta}$, where the $f_{\alpha \beta}$ are holomorphic.\\

Since $Y$ is a transverse analytic subvariety, the image $Y_{\alpha} = \varphi^{-1}_{\alpha}(Y \cap U_\alpha)$ of $Y$ in the chart $U_\alpha$ has the form $Y_\alpha = \pi_\alpha^{-1}( Z_\alpha)$, where $Z_\alpha \subseteq V_\alpha$ is an analytic subvariety of $V_\alpha$ and $\pi_\alpha : U_\alpha = V_\alpha \times W_\alpha \to V_\alpha$ is the projection. By Hironaka's resolution of singularities \cite{hir} (see, eg \cite[Theorem 2.0.1]{wlod} for the case of analytic varieties), there exists a canonical desingularisation $p_\alpha : \tilde{Z}_\alpha \to Z_\alpha$, where $\tilde{Z}_\alpha$ is smooth, $p_\alpha$ is proper, bimeromorphic and an isomorphism over the non-singular locus of $Z_\alpha$. Set $\tilde{Y}_\alpha = \tilde{Z}_\alpha \times W_\alpha$ and let $q_\alpha : \tilde{Y}_\alpha \to Y_\alpha$ be given by $q_\alpha( \tilde{z} , w ) = ( p_\alpha(\tilde{z}) , w)$, where $\tilde{z} \in \tilde{Z}_\alpha$, $w \in W_\alpha$. Observe that the composition $\varphi_\alpha^{-1} \circ \iota \circ q_\alpha : \tilde{Z}_\alpha \times W_\alpha = \tilde{Y}_\alpha \to U_\alpha = V_\alpha \times W_\alpha$ has the form
\[
( \tilde{z} , w ) \mapsto ( p_\alpha(\tilde{z}) , w).
\]
In particular, $\iota \circ q_\alpha$ satisfies (i), (ii) and (iii) (restricted to $U_\alpha$). If we can show that the local desingularisations $\{ q_\alpha : \tilde{Y}_\alpha \to Y_\alpha \}$ glue together on the overlaps of coordinate charts, we will have obtained our desired desingularisation $q : \tilde{Y} \to Y$. In fact, this is easily seen to follow from the {\em functioriality} property of the desingularisations $\tilde{Z}_\alpha \to Z$ \cite[Theorem 2.0.1 (3)]{wlod}, which shows that the $\tilde{Z}_\alpha$ agree on overlaps. In more detail, this means that the change of coordinate maps $f_{\alpha \beta} : Z_\beta \cap V_{\beta \alpha} \to Z_\alpha \cap V_{\alpha \beta}$ lift to $\tilde{f}_{\alpha \beta} : p_\beta^{-1}( V_{\beta \alpha} ) \to p_\alpha^{-1}( V_{\alpha \beta} )$ satisfying an associativity condition on triple overlaps (by functoriality of the $\tilde{f}_{\alpha \beta}$). Let $\tilde{Y}_{\alpha \beta} = p_\alpha^{-1}(V_{\alpha \beta}) \times W_{\alpha \beta} \subseteq \tilde{Z}_\alpha \times W_\alpha = \tilde{Y}_\alpha$. We define transition maps $\psi_{\alpha \beta} : \tilde{Y}_{\beta \alpha} \to \tilde{Y}_{\alpha \beta}$ by
\[
\psi_{\alpha \beta}( \tilde{z} , w ) = ( \tilde{f}_{\alpha \beta}( \tilde{z} ) , g_{\alpha \beta}( p_\beta(\tilde{z} , w) ).
\]
Then it is easy to check that the $\psi_{\alpha \beta}$ satisfy the appropriate associativity condition on triple overlaps (because the $\tilde{f}_{\alpha \beta}$ and the $\varphi_{\alpha \beta} = (f_{\alpha \beta} , g_{\alpha \beta})$ satisfy such conditions), hence allow us to glue the $\{ q_\alpha : \tilde{Y}_\alpha \to Y_\alpha \}$ together to obtain the desired $q : \tilde{Y} \to Y$.
\end{proof}

\subsection{Transverse coherent sheaves and stability}\label{sec:stability}

Let $X$ be a compact oriented, taut, transverse Hermitian foliation of complex codimension $n$ with transverse Gauduchon metric $g$. 

\begin{definition}\label{def:degslopebundle}
Let $E$ be a transverse holomorphic vector bundle of rank $r$ which admits a transverse Hermitian metric $h$. We define the {\em degree} of $E$, denoted $\deg(E)$ to be the real number
\[
\deg(E) = \frac{i}{2\pi} \int_X tr(F_E) \wedge \omega^{n-1} \wedge \chi,
\]
where $F_E$ is the curvature of the Chern connection associated to $h$. This is independent of the choice of transverse Hermitian metric, because $tr(F_E)$ is independent of $h$ up to a $\partial \overline{\partial}$-exact term, which by the basic Stokes' theorem and the Gauduchon property $\partial \overline{\partial} \omega^{n-1} = 0$ does not alter the degree. The {\em slope} of $E$, denoted $\mu(E)$ is defined by $\mu(E) = \deg(E)/r$.
\end{definition}

\begin{remark}
Note that $\deg(E)$ depends on the choice of leafwise volume form $\chi$. Moreover, our argument that $\deg(E)$ is independent of the choice of transverse Hermitian metric only holds in the case that the foliation is taut, since otherwise when applying Stokes' theorem there would be an additional term which in general can change the degree. Let us also point out that the degree of a transverse holomorphic vector bundle $E$ is only defined when $E$ admits a transverse Hermitian metric.
\end{remark}

\begin{definition}\label{def:tcs}
Let $X$ be a foliated manifold with transverse complex structure and let $\mathcal{O}$ denote the sheaf of basic holomorphic functions on $X$. A sheaf of $\mathcal{O}$-modules $\mathcal{F}$ is called a {\em transverse coherent sheaf} if locally, $\mathcal{F}$ is given as the cokernel of a sheaf map $\mathcal{O}^p \to \mathcal{O}^q$ for some $p,q$.
\end{definition}
By this definition, in a local foliated chart $U = V \times W \subset X$, a transverse coherent sheaf is the same thing as a coherent sheaf on $V$. In particular it follows that to any local property of coherent sheaves, there is a corresponding local property for transverse coherent sheaves. In particular, we may speak of {\em torsion free}, {\em reflexive} and {\em locally free} transverse coherent sheaves. It is easy to see that locally free transverse coherent sheaves correspond to transverse holomorphic vector bundles by taking the sheaf of basic holomorphic sections. To any transverse coherent sheaf $\mathcal{F}$, we may associate a determinant $\det({\mathcal{F}})$ which is a transverse holomorphic line bundle. The determinant $\det({\mathcal{F}})$ is constructed exactly as in the non-foliated setting \cite[Chapter V, \textsection 6]{kob}. Adapting the proofs in \cite[Chapter V, \textsection 5]{kob}, we also find that a torsion-free (resp. reflexive) transverse coherent sheaf is locally free outside a transverse analytic subvariety of codimension at least $2$ (resp. $3$).\\ 

In order to define stability of transverse holomorphic vector bundles, we need to define the degree of transverse coherent subsheaves. However, there is a complication due to the fact that if $\mathcal{F}$ is a coherent subsheaf of a transverse holomorphic vector bundle $E$, then even if $E$ admits a transverse Hermitian metric it is not at all clear whether the determinant line bundle $det(\mathcal{F})$ associated to $\mathcal{F}$ admits a transverse Hermitian metric. Thus we can not simply define $\deg(\mathcal{F})$ to be $\deg(det(\mathcal{F}))$. We get around this problem using a foliated resolution of singularities.\\

Let $E$ be a transverse holomorphic vector bundle of rank $r$ which admits a transverse Hermitian metric. Let $s$ be a positive integer less than $r$ and let $q : Gr(s,E) \to X$ be the associated Grassmannian bundle of $E$ whose fibre over $x \in X$ is the Grassmannian $Gr(s,r)$ of $s$-dimensional complex subspaces of $E_x$. We note that $Gr(s,E)$ has a natural taut, Riemannian foliation. To see this, note that the vector bundle $E$ is constructed by patching together local trivialisations over foliated charts of $X$ such that the transition functions $g_{\alpha \beta}$ are basic holomorphic. Then $Gr(s,E)$ is just the associated Grassmannian bundle built out of the transition functions $g_{\alpha \beta}$ by the natural action of $GL(r,\mathbb{C})$ on $Gr(s,r)$. In this way we obtain a lift of the Riemannian foliation on $X$ to a Riemannian foliation on $Gr(s,E)$ which is taut, with leafwise volume form $q^*(\chi)$. Moreover $Gr(s,E)$ has a natural transverse complex structure.\\

On $Gr(s,E)$ we have the tautological rank $s$ transverse holomorphic vector bundle $\mathbb{F} \to Gr(s,E)$, which comes with a natural inclusion $j : \mathbb{F} \to q^*(E)$. A transverse holomorphic subbundle $F \subset E$ of $E$ of rank $s$ is precisely a basic holomorphic section $s : X \to Gr(s,E)$. The correspondence is given by pulling back the tautological bundle, i.e. $F = s^*(\mathbb{F})$. This correspondence extends to rank $s$ coherent subsheaves in the following sense. Consider a closed irreducible transverse analytic subvariety $\iota : Y \to Gr(s,E)$ such that $Y$ generically projects isomorphically onto $X$ via $q$. To such a subvariety $Y$, we associate the transverse coherent sheaf $\mathcal{F} = q_*( \iota^* \mathcal{O}(\mathbb{F}))$, which is naturally a subsheaf of $q_* \iota^* q^* \mathcal{O}(E) \cong \mathcal{O}(E)$. Moreover, since $Y$ is irreducible one finds that the quotient sheaf $\mathcal{O}(E)/\mathcal{F}$ is torsion-free. Conversely, if $\mathcal{F}$ is a rank $s$ subsheaf of $E$ with torsion free quotient then we recover $Y$ as follows. On the complement $U = X \setminus S$ of a closed transverse analytic subvariety $S \subset X$, we have that $\mathcal{F}$ is given by a transverse holomorphic subbundle $F \subset E$ over $U$. Let $s : U \to Gr(s,E)$ be the corresponding section of $Gr(s,E)$ and let $Y$ be the closure of the image of $s$. Since $\mathcal{F} \to E$ was assumed to have torsion free quotient, it is easy to see that we recover $\mathcal{F}$ from $Y$ by taking $q_* \iota^* \mathbb{F}$.\\

We have established a correspondence between rank $s$ coherent subsheaves of $E$ with torsion free quotient and closed irreducible subvarieties $\iota :  Y \to Gr(s,E)$ that generically project isomorphically to $X$. We would like to define the degree of $Y$ to be $\deg( \iota^* \mathbb{F})$, however $Y$ may be singular, so we need to take a resolution of singularities $\alpha : \widetilde{Y} \to Y$. Thus we at last arrive at the following definition:
\begin{definition}\label{def:degreesubsheaf}
Let $E$ be a transverse holomorphic vector bundle which admits a transverse Hermitian metric and let $\mathcal{F} \to E$ be a transverse coherent subsheaf with $0 < s = rk(\mathcal{F}) < r = rk(E)$ and such that the quotient $\mathcal{O}(E)/\mathcal{F}$ is torsion-free. We define the degree of $\mathcal{F}$ as follows. Let $\iota : Y \to Gr(s,E)$ be the corresponding irreducible transverse analytic subvariety of $Gr(s,E)$ and let $\alpha : \widetilde{Y} \to Y$ be a transverse resolution of singularities. We note that $\mathbb{F}$ is a transverse holomorphic subbundle of $q^*(E)$ via the canonical map $j : \mathbb{F} \to q^*(E)$ and thus inherits an induced transverse Hermitian metric. Therefore, $(\iota \circ \alpha)^*( \mathbb{F})$ has a natural transverse Hermitian metric, so has an associated Chern connection with curvature $\widetilde{F}_1$. We then define
\[
\deg(\mathcal{F}) = \frac{i}{2\pi} \int_{\widetilde{Y}} tr( \widetilde{F}_1) \wedge (q \circ \iota \circ \alpha)^*( \omega^{n-1} \wedge \chi).
\]
We define the slope of $\mathcal{F}$ to be $\mu(\mathcal{F}) = \deg(\mathcal{F})/s$.
\end{definition}
\begin{remark}
For this definition to make sense, we need to check that it is independent of the choice of resolution $\alpha : \widetilde{Y} \to Y$. In fact, this will follow from the Proposition \ref{prop:cw} below. Moreover, if $\mathcal{F}$ is actually a transverse holomorphic subbundle of $E$, then it is easy to check that this definition of degree agrees with the previous definition. The definition is also independent of the choice of Hermitian metric on $E$ because the only effect of this is to change the induced Hermitian metric on the line bundle $(\iota \circ \alpha)^*( \mathbb{F} )$, which changes $tr( \widetilde{F}_1 )$ by a $\partial \overline{\partial}$-exact term. By the basic Stokes' theorem and tautness such a term does not contribute to the above integral.
\end{remark}

\begin{proposition}\label{prop:cw}
Let $E$ be a transverse holomorphic vector bundle which admits a transverse Hermitian metric and let $\mathcal{F} \to E$ be a transverse coherent subsheaf with $0 < s = rk(\mathcal{F}) < r = rk(E)$ and such that the quotient $\mathcal{O}(E)/\mathcal{F}$ is torsion-free. Then since $\mathcal{F}$ and $\mathcal{O}(E)/\mathcal{F}$ are torsion free, there is a transverse analytic subvariety $S \subset X$ of complex codimension at least $2$ such that on $X \setminus S$, we have that $\mathcal{F}$ and $\mathcal{O}(E)/\mathcal{F}$ are locally free and hence $\mathcal{F}$ is given by a transverse holomorphic subbundle $F \to E$ on $X \setminus S$. Let $h_1$ be the transverse Hermitian metric on $F$ obtained by restriction to $F$ of the transverse Hermitian metric on $E$. Let $F_1$ denote the curvature $2$-form of the Chern connection on $F$ associated to $h_1$. Then $F_1$ is in $L^1_B$ on $X \setminus S$ and
\[
{\rm deg}(\mathcal{F}) = \frac{i}{2\pi} \int_{X \setminus S} tr( F_1 ) \wedge \omega^{n-1} \wedge \chi.
\]
In particular, this implies that ${\rm deg}(\mathcal{F})$ is independent of the choice of resolution of singularities $\alpha : \widetilde{Y} \to Y$.
\end{proposition}
\begin{proof}
Let $l : \widetilde{Y} \to X$ be the composition $l = q \circ \iota \circ \alpha$ and let $\widetilde{S} = l^{-1}(S)$. Then $l : \widetilde{Y} \setminus \widetilde{S} \to X \setminus S$ is a diffeomorphism and $\widetilde{S}$ has measure zero in $\widetilde{Y}$. The result now follows by noting that $\widetilde{F}_1$ in Definition \ref{def:degreesubsheaf} is related to $F_1$ on the complement of $S$ by $\widetilde{F}_1 = l^*(F_1)$.
\end{proof}

\begin{definition}\label{def:stability}
Let $E$ be a transverse holomorphic vector bundle which admits a transverse Hermitian metric. We say that $E$ is {\em stable} (resp. {\em semistable}) if for each transverse coherent subsheaf $\mathcal{F} \to E$ with $0 < s = rk(\mathcal{F}) < r = rk(E)$ and such that the quotient $\mathcal{O}(E)/\mathcal{F}$ is torsion-free, we have
\[
\mu(\mathcal{F}) < \mu(E) \quad ({\rm resp.} \; \; \mu(\mathcal{F}) \le \mu(E) ).
\]
We also say that $E$ is {\em polystable} if $E$ the direct sum of stable bundles of the same slope.
\end{definition}

Note that if $E$ is stable and $F \subset E$ is a proper, non-trivial subbundle then $E/F$ is torsion free, hence $\mu(F) < \mu(E)$ by stability.

\begin{definition}
A transverse holomorphic bundle $E$ is called {\em simple} if the only basic holomorphic endomorphisms $E \to E$ are constant multiples of the identity, that is, $E$ is simple if $H^0_B( X , End(E) ) = \mathbb{C} Id$.
\end{definition}

\begin{proposition}\label{prop:stableissimple}
Let $E$ be a stable transverse holomorphic vector bundle. Then $E$ is simple.
\end{proposition}
\begin{proof}
Let $f : E \to E$ be a basic holomorphic endomorphism of $E$. The coefficients of the characteristic polynomial of $f$ are basic holomorphic functions on $X$ and therefore constant since $X$ is assumed to be compact. Therefore $f$ has constant eigenvalues and we can decompose $E$ into the generalised eigenspaces of $f$:
\[
E = \bigoplus_{\lambda} E_\lambda.
\]
If $f$ has more than one distinct eigenvalue, then at least one summand $E_\lambda$ will have $\mu(E_\lambda) \ge \mu(E)$, contradicting stability. Therefore $f$ can only have one eigenvalue, say $\lambda$. Let $g = f - \lambda Id_E$. Then $g : E \to E$ is nilpotent. To complete the proof, it suffices to show $g = 0$. Suppose $g \neq 0$. Then $Ker(g)$ is a proper, non-trivial subbundle of $E$ and so $\mu( Ker(g) ) < \mu(E)$ by stability of $E$. But this implies $\mu( Im(g) ) = \mu( E/Ker(g) ) > \mu(E)$. If $g$ is a non-zero nilpotent endomorphism of $E$, then $Im(g)$ is a proper, non-trivial subbundle of $E$ contradicting stability. Thus $g = 0$.
\end{proof}

\section{Transverse Hermitian-Einstein connections}\label{sec:trhe}

Let $X$ be a compact oriented, taut, transverse Hermitian foliation of complex codimension $n$ with transverse Gauduchon metric $g$. 

\begin{definition}\label{def:he}
\leavevmode
\begin{itemize}
\item[(i)]{Let $E$ be a transverse Hermitian bundle. A basic unitary connection $A$ on $E$ is called {\em transverse Hermitian-Einstein} if its curvature $2$-form $F_A$ is of type $(1,1)$ and satisfies
\[
i \Lambda F_A = \gamma_A id_E,
\]
for some real constant $\gamma_A$, called the {\em Einstein factor} of $A$.}
\item[(ii)]{Let $E$ be a transverse holomorphic bundle. A transverse Hermitian metric is called a {\em transverse Hermitian-Einstein metric} if the associated Chern connection is transverse Hermitian-Einstein.}
\end{itemize}
\end{definition}

\begin{definition}
Let $E$ be a transverse Hermitian bundle. A basic unitary connection $A$ on $E$ is called {\em irreducible} if the only $A$-covariantly constant sections of $End(E)$ are the constant multiples of the identity.
\end{definition}

\begin{proposition}\label{prop:einsteinfactor}
Let $A$ be a transverse Hermitian-Einstein connection on $E$. Then the Einstein factor of $A$ is given by
\[
\gamma_A = \frac{2\pi}{(n-1)! Vol(X) } \mu(E),
\]
where $Vol(X) = \int_X \frac{\omega^n}{n!} \wedge \chi.$
\end{proposition}
\begin{proof}
Same as in the non-foliated setting. 
\end{proof}
Note that by this proposition $\gamma_A$ only depends on $E$ and not on the connection $A$. Thus we will often denote the Einstein factor as $\gamma_E$ and call it the Einstein factor of $E$.

\begin{theorem}\label{thm:vanishing}
Let $E$ be a transverse holomorphic bundle admitting a transverse Hermitian-Einstein metric $h$ with Chern connection $A$. If $\deg(E)$ is negative, then $E$ has no global basic holomorphic sections. If $\deg(E)$ is zero, then every global basic holomorphic section of $E$ is $d_A$-constant.
\end{theorem}
\begin{proof}
Let $s$ be a basic holomorphic section of $E$. Then one finds:
\begin{equation*}
\begin{aligned}
P( h(s,s) ) &= i\Lambda \overline{\partial} \partial h(s,s) \\
&= i\Lambda h( \overline{\partial}_E \partial_E s , s) - i\Lambda h( \partial_E s , \partial_E s) \\
&= i\Lambda h( F_E s , s) - | \partial_E s|^2 \\
&= \gamma_E |s|^2 - |\partial_E s|^2.
\end{aligned}
\end{equation*}
The maximum principle now implies that if $\gamma_E < 0$, then $s= 0$ and if $\gamma_E = 0$, then $\partial_E s = 0$. But in this case $s$ is holomorphic, so $\overline{\partial}_E s = 0$ and hence $\nabla_E s = 0$.
\end{proof}

\begin{proposition}\label{prop:unique}
Let $E$ be a transverse holomorphic bundle which is simple. If a transverse Hermitian-Einstein metric exists on $E$, it is unique up to rescaling by a positive constant.
\end{proposition}
\begin{proof}
Let $h_1,h_2$ be two transverse Hermitian-Einstein metrics on $E$. The identity on $E$, viewed as a map $Id : (E,h_1) \to (E,h_2)$ is a holomorphic section of $Hom( (E,h_1) , (E,h_2))$. But $h_1$ and $h_2$ induce a Hermitian-Einstein metric on $Hom( (E,h_1) , (E, h_2))$, so by Theorem \ref{thm:vanishing} we have that $Id : (E , h_1) \to (E , h_2)$ is covariantly constant (since $\deg( End(E) ) = 0$). This means that the unitary connections associated to $h_1$ and $h_2$ are related by $Id$, so they are equal. In particular, this gives $\partial_{E,h_1} = \partial_{E , h_2}$. Now let $f : E \to E$ be the unique self-adjoint endomorphism of $E$ for which $h_2(s,t) = h_1( f(s) , t)$. The equality $\partial_{E,h_1} = \partial_{E,h_2}$ implies $\partial_{E,h_1} f = 0$ and since $f$ is self-adjoint, we also get $\overline{\partial}_E f = 0$. But if $E$ is simple, this implies that $f$ is a multiple of the identity and hence $h_1$ and $h_2$ are related by a constant rescaling.
\end{proof}

Next we have the following transverse version of the Bogomolov inequality \cite{bog}:
\begin{theorem}
Let $E$ be a transverse Hermitian bundle of rank $r$ admitting a transverse Hermitian-Einstein connection and suppose that $n \ge 2$. Then we have the following inequality
\begin{equation}\label{equ:bog}
\int_X \left( 2r \cdot c_{2,B}(E) - (r-1) \cdot c_{1,B}^2(E) \right) \wedge \omega^{n-2} \wedge \chi \ge 0,
\end{equation}
where $c_{1,B}(E), c_{2,B}(E)$ are the basic Chern forms of degree $1$ and $2$. That is, for any transverse Hermitian metric $h$ with associated Chern connection $A$, we define
\[
c_{1,B}(E,A) = \frac{i}{2\pi} tr(F_A), \quad c_{2,B}(E,A) = -\frac{1}{8\pi^2} \left( (tr(F_A))^2 - tr(F_A^2) \right),
\]
where $F_A$ is the curvature of $A$. The basic Chern forms are independent of the choice of transverse Hermitian metric $h$ up to $\partial \overline{\partial}$-exact terms, so if $X$ is Gauduchon, the left hand side of (\ref{equ:bog}) is independent of the choice of $h$. Moreover, equality holds if and only if $A$ is projectively flat.
\end{theorem}
\begin{proof}
The inequality (\ref{equ:bog}) is obtained by integrating over $X$ a pointwise inequality
\[
\left( 2r \cdot c_{2,B}(E) - (r-1) \cdot c_{1,B}^2(E) \right) \wedge \omega^{n-2} \wedge \chi \ge 0,
\]
where $\ge 0$ means that the left hand side is a non-negative multiple of $dvol_X$. This pointwise inequality is obtained by a local computation which is no different than in the non-foliated setting \cite[Theorem 2.2.3]{lt}.
\end{proof}

\begin{theorem}\label{thm:heispolystable}
Let $E$ be a transverse holomorphic vector bundle on $X$ which admits a transverse Hermitian-Einstein metric. Then $E$ is polystable.
\end{theorem}
\begin{proof}
Let $\mathcal{F}$ be a coherent subsheaf of $E$ with $0 < s = rk(\mathcal{F}) < r = rk(E)$ and with torsion-free quotient. Then since the quotient is torsion free, there is some transverse analytic subvariety $S \subset X$ of complex codimension at least $2$ such that on the complement $X \setminus S$, the quotient is a vector bundle and therefore we also have that $\mathcal{F}$ is given by a holomorphic subbundle $F \subset E$ on $X \setminus S$. Let $\pi \in \mathcal{C}^\infty_B( X \setminus S , End(E) )$ be the orthogonal projection from $E$ to $F$, which is defined on $X \setminus S$. Let $F_1$ denote the curvature of the Chern connection on $F$ induced by the inclusion $F \to E$. Then we have (see \cite{gri}):
\[
i \Lambda tr( F_1 ) =  i\Lambda tr( \pi F_E \pi ) - | \partial_{End(E)} \pi |^2.
\]
Note that $tr(F_1)$ is in $L^1_B$ by Proposition \ref{prop:cw}. Wedging with $\omega^n \wedge \chi$ and integrating over $X$, we obtain:
\[
in \int_X tr(F_1) \wedge \omega^{n-1} \wedge \chi =  \int_X tr( \pi i \Lambda(F_E) \pi ) \wedge \omega^{n} \wedge \chi -  || \partial_{End(E)} \pi ||^2_{L^2}.
\]
Using Proposition \ref{prop:cw} and the fact that $E$ is Hermitian-Einstein, $i \Lambda F_E = \gamma_E Id_E$, we get
\begin{equation*}
\begin{aligned}
2\pi n \, \deg(\mathcal{F}) &= \int_X \gamma_E \, tr( \pi) \wedge \omega^n \wedge \chi -  || \partial_{End(E)} \pi ||^2_{L^2} \\
&= \gamma_E \, rk(\mathcal{F}) n! Vol(X) -  || \partial_{End(E)} \pi ||^2_{L^2} \\
&= 2 \pi n \, rk(\mathcal{F}) \mu(E) -  || \partial_{End(E)} \pi ||^2_{L^2},
\end{aligned}
\end{equation*}
where in the last line we used Proposition \ref{prop:einsteinfactor}. This shows that $\mu(\mathcal{F}) \le \mu(E)$ and that equality occurs if and only if $\partial_{End(E)} \pi = 0$. However, in this case since $\pi$ is self-adjoint, we also get that $\overline{\partial}_{End(E)} \pi = 0$, so $\pi$ is covariantly constant. In particular, $\pi$ is a basic holomorphic section of $End(E)$ defined on $X \setminus S$, but since $S$ has codimension at least $2$, Hartog's theorem implies that $\pi$ extends to a basic holomorphic section of $End(E)$ on $X$. The identities $\pi^2 = \pi$ and $\pi^* = \pi$ continue to hold on all of $X$, hence $\pi$ extends to $X$ as a projection. It follows that $F$ extends to a holomorphic subbundle on all of $E$. Moreover, since $\pi$ is covariantly constant, the orthogonal complement $F^\perp$ is also a holomorphic subbundle of $E$. We thus get an orthogonal, holomorphic splitting $E = F \oplus F^\perp$, where in addition $\mu(F) = \mu(F^\perp) = \mu(E)$. By iterating the above argument, we eventually get a decomposition of $E$ into a direct sum of stable transverse holomorphic bundles each having the same slope, i.e. $E$ is polystable.
\end{proof}

\begin{corollary}\label{cor:irredsimple}
Let $E$ be a transverse holomorphic vector bundle on $X$ which admits a transverse Hermitian-Einstein metric with Chern connection $A$. Then $A$ is irreducible if and only if $E$ is simple.
\end{corollary}
\begin{proof}
If $E$ is simple, then clearly $A$ is irreducible. Conversely if $E$ is transverse Hermitian-Einstein and irreducible, then $E$ is polystable by Theorem \ref{thm:heispolystable}. However, if $E$ is polystable but not stable, then clearly there would be covariantly constant sections of $End(E)$ which are not multiples of the identity, so $E$ is stable. But we have already shown in Proposition \ref{prop:stableissimple} that stable implies simple.
\end{proof}

\section{The transverse Hitchin-Kobayashi correspondence}\label{sec:trhk}

In this section we will prove that every stable transverse holomorphic vector bundle which admits a transverse Hermitian metric admits a transverse Hermitian-Einstein metric. Our proof will be an adaptation to the foliated setting of the proof of the usual Hitchin-Kobayashi correspondence given by Uhlenbeck-Yau \cite{uy}, using the method of continuity. A thorough treatment of the Uhlenbeck-Yau proof, adapted to the case of Gauduchon metrics is given in the book of L\"ubke and Teleman \cite{lt}. In what follows we will outline the main steps in the Uhlenbeck-Yau proof adapted to the foliated setting. We will mostly focus attention on the necessary changes required to adapt the proof to the foliated setting. We will omit details of the proof whenever they are essentially the same as in the non-foliated setting, referring the reader to the relevant sections of \cite{lt}.\\

Let $E$ be a stable transverse holomorphic vector bundle of rank $r$ which admits a transverse Hermitian metric. Let $h_0$ be a fixed choice of a transverse Hermitian metric. Let $d_0 = \partial_0 + \overline{\partial}$ be the associated Chern connection, $F_0$ the curvature of $d_0$ and $K_0 = i\Lambda F_0$ the mean curvature. Any transverse Hermitian metric $h$ on $E$ has the form $h(s,t) = h_0(fs,t)$ for a uniquely determined basic $h_0$-self-adjoint, positive definite endomorphism $f$. Conversely any such endomorphism $f$ determines a transverse Hermitian metric this way. The Chern connection associated to $h$ is $d_0 + f^{-1} \partial_0(f)$ and therefore
\[
K = K_0 + i \Lambda \left( \overline{\partial} ( f^{-1} \partial_0(f) ) \right).
\]
The Hermitian-Einstein equation for $h$ is therefore
\[
K_0 - \gamma Id_E + i \Lambda \left( \overline{\partial} ( f^{-1} \partial_0(f) ) \right) = 0,
\]
where $\gamma$ is the Einstein factor for $E$. Following Uhlenbeck-Yau we will find a solution of this equation by the continuity method. For a real number $\epsilon \in [0,1]$ consider the perturbed equation
\begin{equation}\label{eq:deformed}
L_\epsilon(f) = K_0 - \gamma Id_E + i \Lambda \left( \overline{\partial} ( f^{-1} \partial_0(f) ) \right) + \epsilon \cdot {\rm log}(f) =0.
\end{equation}
We first show the equation $L_\epsilon(f) = 0$ has a solution for all $\epsilon \in (0 , 1]$. More precisely, we will show there is a solution $f_1$ to $L_1(f_1) = 0$. We then let $J \subset (0,1]$ be defined as the set of $\epsilon \in (0,1]$ for which there is a map $f : [\epsilon , 1 ] \to Herm^+_B(E,h_0)$ such that $f(1) = f_1$ and $L_{\epsilon'}(f_{\epsilon'}) = 0$ for all $\epsilon' \in [\epsilon , 1]$ (c.f. \cite[\textsection 3.1]{lt}). From this definition, $J$ is an interval containing $1$. We will show that $J$ is open and closed in $(0,1]$ and therefore $J = (0,1]$. In particular, for every $\epsilon \in (0,1]$ we will have obtained a solution $f_\epsilon$ of the equation $L_\epsilon(f_\epsilon) = 0$. Next one considers the limit $\displaystyle{\lim_{\epsilon \to 0}} \; f_\epsilon$. If the limit $f_0 = \displaystyle{\lim_{\epsilon \to 0}} \; f_\epsilon$ exists, one shows that $h(s,t) = h_0(f_0 s , t)$ is a transverse Hermitian-Einstein metric. On the other hand if the limit does not exist, we will construct a transverse coherent sheaf violating the stability condition for $E$.

\subsection{Existence of $f_1$}\label{sec:f1}

We follow \cite[Lemma 3.2.1]{lt}. We will show that there is a transverse Hermitian metric $h_0$ on $E$ for which a solution $f_1$ to $L_1(f_1) = 0$ exists. First choose any transverse Hermitian metric $h$, let $K_h$ be the mean curvature and set $K^0_h = K_h - \gamma Id_E$. From the definitions of $\gamma$ and $K_h$, it follows that
\[
\int_X tr( K^0_h) \omega^n \wedge \chi = 0
\]
so that $tr(K^0_h)$ is $L^2$-orthogonal to the constant functions $\mathbb{C} = ker(P^*)$, hence $tr(K^0_h) \in im(P)$. So there exists a basic function $\varphi$ satisfying
\[
P(\varphi) = -\frac{1}{r} tr(K^0_h).
\]
Moreover, we may assume $\varphi$ is real-valued since $P$ and $tr(K^0_h)$ are real. Now define a new transverse Hermitian metric by $h_1 = e^\varphi h$. One finds $K^0_{h_1} = K^0_h + P(\varphi) Id_E$ and therefore $tr(K^0_{h_1}) = tr(K^0_h) + rP(\varphi) = 0$. Now let $h_0$ be the transverse Hermitian metric defined by
\[
h_0(s,t) = h_1( {\rm exp}( K_{h_1} )s , t ),
\]
which is a Hermitian metric because $K_{h_1}$ is Hermitian with respect to $h_1$. Note that $tr(K^0_{h_1}) = 0$ implies that $tr(K^0_{h_0}) = 0$ as well. Let $f_1 = {\rm exp}(-K^0_{h_1})$, which is positive and self-adjoint with respect to $h_0$, hence also with respect to $h_1$. Letting $L_\epsilon$ be defined with respect to $h_0$ as in Equation (\ref{eq:deformed}), we have
\[
L_1(f_1) = K_{h_1} - \gamma Id_E  - K^0_{h_1} = K^0_{h_1} - K^0_{h_1} = 0.
\]
Thus we have found a transverse Hermitian metric $h_0$ and a solution $f_1$ to $L_1(f_1) = 0$. From now on, we keep $h_0$ as our fixed choice of transverse Hermitian metric.

\subsection{Continuity method: $J$ is open}\label{sec:jopen}

Following \cite[Lemma 3.2.3]{lt}, we introduce an operator $\hat{L}(\epsilon , f)$, where $\epsilon \in [0,1]$ and $f \in Herm^+_B(E , h_0)$ by
\[
\hat{L}(\epsilon , f ) = f \circ L_\epsilon(f).
\]
Clearly $\hat{L}(\epsilon , f) = 0$ if and only if $L_\epsilon(f) = 0$. The advantage of using $\hat{L}(\epsilon , f)$ is that it takes values in $Herm_B(E , h_0)$ (the proof is the same as \cite[Lemma 3.2.3]{lt}). Since $\hat{L}$ is a second order differential operator which depends continuously on $\epsilon$, we may view it as a continuous map
\[
\hat{L} : (0,1] \times L^{p,k}_B ( Herm^+_B(E,h_0) ) \to L^{p,k-2}_B( Herm_B(E,h_0) ),
\]
for all sufficiently large $k$ so that the appropriate Sobolev multiplication theorems hold. Here we define $L^{p,k}_B ( Herm^+_B(E,h_0) )$ to be the interior of the closure of $Herm^+_B(E,h_0)$ in $L^{p,k}_B( Herm_B(E,h_0) )$. For all large enough $k$, the map $\hat{L}$ is moreover differentiable. Let $d_2 \hat{L}$ denote the derivative in the $f$-direction. Clearly this is a linear second order differential operator $d_2 \hat{L}(\epsilon , f) : L^{p,k}_B( Herm_B(E , h_0) ) \to L^{p,k-2}_B( Herm_B(E,h_0))$. It is easy to see that the symbol of $d_2 \hat{L}(\epsilon , f)$ agrees with the symbol of $P_{End(E)}$, which in turn is the symbol of the self-adjoint transverse elliptic operator $\partial_{End(E)}^* \partial_{End(E)}$. It follows that $d_2 \hat{L}(\epsilon , f)$ is transverse elliptic and has index zero, since the difference between $d_2 \hat{L}(\epsilon , f)$ and $\partial_{End(E)}^* \partial_{End(E)}$ is a first order operator, which is compact when regarded as an operator $L^{p,k}_B \to L^{p,k-2}_B$, by Sobolev compactness.\\

We may now argue that $J$ is open as follows (see \cite[Corollary 3.2.7]{lt}).
\begin{lemma}\label{lem:regularity1}
For all large enough $k$, the following holds: if $f$ is an $L^{p,k}_B$-solution of $L_\epsilon(f) = 0$, then $f$ is smooth and basic.
\end{lemma}
\begin{proof}
The equation $L_\epsilon(f) = 0$ can be written in the form
\[
P(f) = \{ f , \log(f) , \partial_{End(E)} f , \overline{\partial}_{End(E)} f \},
\]
where $P = i \Lambda \overline{\partial}_{End(E)} \partial_{End(E)}$ and $\{ \cdots \}$ is some multilinear algebraic expression in its arguments. This means that for all sufficiently large $k$, $f \in L^{p,k}_B$ implies $P(f) \in L^{p,k-1}_B$. By Lemma \ref{lem:trreg}, this implies $f \in L^{p,k+1} \cap L^{p,1}_B$. By repeated application of Lemma \ref{lem:trreg} and the ordinary (i.e. non-basic) Sobolev embedding theorem, we have that $f$ is smooth. But $f \in L^{p,1}_B$ implies $\nabla_V (f) = 0$, where $\nabla_V$ is defined as in Lemma \ref{lem:trreg}. Hence $f$ is basic.
\end{proof}
Using this lemma and the Banach space implicit function theorem, to prove that $J$ is open it suffices to show that if $\hat{L}(\epsilon , f) = 0$, then $d_2 \hat{L}(\epsilon , f) : L^{2,k}_B( Herm_B(E , h_0) ) \to L^{2,k-2}_B( Herm_B(E,h_0))$ is an isomorphism of Banach spaces. We have already seen that $d_2 \hat{L}(\epsilon , f)$ is transverse elliptic with index zero. Hence it suffices to show that if $\hat{L}(\epsilon , f) = 0$ for some $\epsilon \in (0,1]$ and $f \in Herm_B^+(E,h_0)$, then the kernel of $d_2 \hat{L}(\epsilon , f)$ is trivial. Thus, suppose that $L_\epsilon(f) = 0$ and that $\varphi \in Herm_B(E , h_0)$ satisfies $d_2 \hat{L}(\epsilon , f) \varphi = 0$ (note that by an argument similar to the proof of Lemma \ref{lem:regularity1}, any element of the kernel of $d_2 \hat{L}(\epsilon , f)$ is smooth). By adapting \cite[Proposition 3.2.5]{lt} to the foliated setting, we obtain an inequality $P( |\eta |^2 ) + 2\epsilon |\eta|^2 \le 0$, where $\eta = f^{-1/2} \circ \varphi \circ f^{-1/2}$. By the maximum principle for basic functions, we find $|\eta| = 0$ and thus $\varphi  = 0$. We note that the proof of \cite[Proposition 3.2.5]{lt} makes use of local diagonalisation of self-adjoint endomorphisms of $E$ by unitary frames on a dense open subset (see \cite[Section 7.4]{lt}). One can check that these results extend in a straighforward manner to the foliated setting.

\subsection{Continuity method: $J$ is closed}\label{sec:jclosed}

Given the initial solution $f_1$ constructed in Section \ref{sec:f1}, we have that the equation $L_\epsilon(f) = 0$ has a unique solution on a maximal open interval $J = (\epsilon_0 , 1]$ for some $\epsilon_0 \ge 0$ (uniqueness follows from the Banach space implicit function theorem as in Section \ref{sec:jopen}). Let us also note that from the Banach space implicit function theorem we have that $f_\epsilon$ is differentiable in $\epsilon$. By \cite[Lemma 3.2.1]{lt}, which is easily seen to carry over to the foliated setting, we can assume that $det(f_\epsilon) = 1$ for all $\epsilon$. We will show that if $\epsilon_0 > 0$, then $f_{\epsilon_0} = \displaystyle{\lim_{\epsilon \to \epsilon_0}} \; f_\epsilon$ exists and is a solution of $L_{\epsilon_0}( f_{\epsilon_0}) = 0$. But the results of Section \ref{sec:jopen} would imply that there exists a solution on a larger open interval, contradicting maximality of $J$. Therefore, it follows that $J = (0,1]$.\\

Following \cite[Section 3.3]{lt}, we define
\[
m_\epsilon = {\rm max}_X | {\rm log}(f_\epsilon) |, \quad \varphi_\epsilon = \frac{d f_\epsilon}{d \epsilon}, \quad \eta_\epsilon = f_\epsilon^{-1/2} \circ \varphi_\epsilon \circ f_\epsilon^{-1/2}.
\]
The estimates in \cite[Section 3.3]{lt} carry over without difficulty to the foliated setting. In particular, we obtain the following results:
\begin{lemma}[\cite{lt}, Lemma 3.3.4]\label{lem:334}
Let $f \in Herm^+_B(E, h_0)$ satisfy $L_\epsilon(f) = 0$ for some $\epsilon > 0$. Then
\begin{itemize}
\item[(i)]{$\frac{1}{2} \cdot P( | {\rm log}(f)|^2) + \epsilon | {\rm log}(f)|^2 \le |K^0| \cdot | {\rm log}(f)|$.}
\item[(ii)]{$m_\epsilon \le \frac{1}{\epsilon} {\rm max}_X |K^0|$.}
\item[(iii)]{$m_\epsilon \le C \cdot \left( || {\rm log}(f) ||_{L^2} + {\rm max}_X |K^0| \right)^2$ where the constant $C$ depends on $g$ and $h_0$.}
\end{itemize}
\end{lemma}

\begin{proposition}[\cite{lt}, Proposition 3.3.5]\label{prop:335}
Suppose there is a positive real number $m$ such that $m_\epsilon \le m$ for all $\epsilon \in (\epsilon_0 , 1]$. Then for all $p > 1$ and $\epsilon \in (\epsilon_0 , 1]$ there is a constant $C(m)$ depending only on $m$ such that:
\begin{itemize}
\item[(i)]{$|| \varphi_{\epsilon} ||_{p,2} \le C(m) \cdot \left( 1 + ||f_\epsilon ||_{p,2} \right)$.}
\item[(ii)]{$ ||f_\epsilon ||_{p,2} \le e^{C(m)(1-\epsilon)} \cdot \left( 1 + ||f_1 ||_{p,2} \right)$.}
\end{itemize}
\end{proposition}

From this proposition we deduce the following:
\begin{proposition}[\cite{lt}, Proposition 3.3.6]
\leavevmode
\begin{itemize}
\item[(i)]{$J = (0,1]$.}
\item[(ii)]{If there is a constant $C$ such that $|| {\rm log}(f_\epsilon)||_{L^2} \le C$ for all $\epsilon \in (0,1]$, then there exists a solution $f_0$ of the equation $L_0(f_0) = 0$ and hence there exists a transverse Hermitian-Einstein metric on $E$.}
\end{itemize}
\end{proposition}
\begin{proof}
i) Suppose that $(\epsilon_0 , 1] \subseteq J$ for some $\epsilon_0 > 0$. We show that the solution $f_\epsilon$ on $(\epsilon_0 , 1]$ extends to $[\epsilon_0 , 1]$. This together with the fact that $J$ is open in $(0,1]$ implies $J = (0,1]$. Choose a $p$ with $p > 2n$. Set $m = \frac{1}{\epsilon_0} \cdot {\rm max}_X |K^0|$. From Lemma \ref{lem:334} (ii) we have $m_\epsilon \le m$, and therefore Proposition \ref{prop:335} (ii) implies
\[
|| f_\epsilon ||_{p,2} \le C(m)
\]
uniformly in $\epsilon > \epsilon_0$. Since $L^{p,2}_B( Herm_B(E,h_0)) )$ is a closed subspace of the reflexive Banach space $L^{p,2}( Herm(E,h_0) )$, it follows that $L^{p,2}_B( Herm_B(E,h_0) )$ is itself a reflexive Banach space. The Banach-Alaogu theorem implies that since the $f_\epsilon$ are uniformly bounded, we can find a subsequence $\{ f_{\epsilon_i} \}_{i \in \mathbb{N}}$ which converges weakly in $L^{p,2}$ to some $f_{\epsilon_0} \in L^{p,2}_B( Herm_B(E,h_0) )$. The bound $m_\epsilon \le m$ implies that the eigenvalues of $f_{\epsilon_0}$ lie in the interval $[ e^{-m} , e^m ]$, hence $f_{\epsilon_0} \in L^{p,2}_B( Herm^+_B(E,h_0))$. By Sobolev compactness of $L^{p,2}_B \to L^{p,1}_B$, we may assume that $f_{\epsilon_i}$ converges to $f_{\epsilon_0}$ strongly in $L^{p,1}_B$. It remains to show that $L_{\epsilon_0}(f_{\epsilon_0}) = 0$ since then by an argument similar the the proof of Lemma \ref{lem:regularity1} ensures that $f_{\epsilon_0}$ is smooth (Lemma \ref{lem:regularity1} requires one to take $k$ sufficiently large. Here we have $k=2$, but since we are assuming $p > 2n$, one sees that $k=2$ is already large enough). This can be shown in exactly the same way as done in \cite[Proposition 3.3.6]{lt}. For (ii), if $|| {\rm log}(f_{\epsilon}) ||_{L^2} \le C$ for all $\epsilon \in (0,1]$ then Lemma \ref{lem:334} (iii) implies a uniform bound for $m_\epsilon$ and then by Proposition \ref{prop:335} (ii) a uniform bound on the $|| f_\epsilon ||_{p,2}$ on $(0,1]$. A similar argument to part (i) of the proof now gives convergence of the $f_\epsilon$ to a solution $f_0$.
\end{proof}

\subsection{Construction of a destabilising subsheaf}

We have seen in Section \ref{sec:jclosed} that a uniform bound on $|| {\rm log}(f_\epsilon) ||_{L^2}$ for all $\epsilon \in (0,1]$ implies the existence of a solution $f_0$ to $L_0(f_0) = 0$ and thus the existence of a transverse Hermitian-Einstein metric on $E$. To prove the existence of a solution $f_0$, it suffices to prove the following:
\begin{proposition}[\cite{lt} Proposition 3.4.1]\label{prop:destab}
If $\limsup\limits_{\epsilon \rightarrow0} || {\rm log}(f_\epsilon) ||_{L^2} = \infty$, then $E$ is not stable.
\end{proposition}

As in the proof of Uhlenbeck-Yau this is done by constructing from $\{ f_\epsilon \}$ a transverse coherent subsheaf $\mathcal{F}$ of $\mathcal{O}(E)$ which violates stability, that is, $\mu(\mathcal{F}) \ge \mu(E)$. The destabilising subsheaf is constructed using the notion of weakly holomorphic subbundle as in \cite{uy}. Here we introduce the corresponding notion in the foliated setting:

\begin{definition}\label{def:weaklyholo}
Let $E$ be a transverse holomorphic vector bundle which admits a transverse Hermitian metric. A {\em transverse weakly holomorphic subbundle} of $E$ is an element $\pi \in L^{2,1}_B( End(E) )$ such that the following identities hold in $L^1_B( End(E) )$:
\[
\pi^* = \pi = \pi^2, \quad \quad (id_E - \pi) \circ \overline{\partial}( \pi ) = 0.
\]
\end{definition}
Observe that if $F \subseteq E$ is a transverse holomorphic subbundle of $E$, then the orthogonal projection $\pi : E \to E$  to $F$ is a transverse weakly holomorphic subbundle of $E$. Hence Definition \ref{def:weaklyholo} generalises the notion of a transverse holomorphic subbundle. Conversely, if $\pi$ is a transverse weakly holomorphic subbundle such that $\pi$ is $\mathcal{C}^\infty$, then the image of $\pi$ defines a transverse holomorphic subbundle of $E$.\\

In \cite{uy, lt}, it is shown that a weakly holomorphic subbundle of $E$ can be represented by a coherent subsheaf of $E$. We explain below how this result can be extended to the foliated setting.
\begin{theorem}\label{thm:twh}
Let $\pi$ be a transverse weakly holomorphic subbundle of $E$. Then there is a transverse coherent sheaf $\mathcal{F}$ and an transverse analytic subset $S \subset X$ such that:
\begin{itemize}
\item[(i)]{The complex codimension of $S$ in $X$ is at least $2$.}
\item[(ii)]{The restriction of $\pi$ to $X \setminus S$ is smooth and therefore defines a transverse holomorphic subbundle $F \subseteq E|_{X \setminus S}$.}
\item[(iii)]{The restriction of $\mathcal{F}$ to $X \setminus S$ is the sheaf of basic holomorphic sections of $F$.}
\end{itemize}
\end{theorem}
\begin{proof}
We adapt the proof given in \cite[S290-S292]{uy} to the the foliated setting. First note that since $\pi = \pi^* = \pi^2$ in $L^{2,1}_B$, then $\pi$ is represented almost everywhere by an orthogonal projection, so the function $tr(\pi)$ is almost everywhere an integer. Since $tr(\pi) \in L^{2,1}_B$ this integer must be a constant, hence $tr(\pi) = s$ for some $s \in \{ 0 , 1 , \dots , r\}$. Therefore, in every foliated coordinate chart $X \supseteq U = V \times W$ over which $E$ is trivialised, $\pi$ determines an $L^{2,1}$ function $\pi|_U : V \to Gr( s , r )$ taking values in the Grassmannian $Gr(s , r)$ of $s$-dimensional subspaces of $\mathbb{C}^r$. The condition $(id_E - \pi) \circ \overline{\partial}( \pi) = 0$ in $L^1_B$ means that the local map $\pi|_U : V \to Gr(s,r)$ is weakly holomorphic in the sense of \cite[S284]{uy}. We then use the fact that a weakly holomorphic map into a projective algebraic manifold (in this case, $Gr(s,r)$) is meromorphic \cite[Theorem 6.1]{uy}. As explained in \cite[S290]{uy}, the collection of locally defined meromorphic maps $\pi|_U : V \to Gr(s,r)$ determined by local holomorphic trivialisations of $E$ defines a coherent sheaf $\mathcal{F}$ (or in our case a transverse coherent sheaf) such that around points where the map $\pi|_U$ is holomorphic, $\mathcal{F}$ is the sheaf of basic holomorphic sections of the pullback under $\pi|_U$ of the tautological bundle on $Gr(s,r)$. Moreover, as $Gr(s,r)$ is projective algebraic, we may assume that the meromorphic maps $\pi|_U$ are holomorphic outside a transverse analytic subset $S \subset X$ of complex codimension at least $2$.
\end{proof}

\begin{remark}
Note that transverse coherent sheaf $\mathcal{F}$ constructed in Theorem \ref{thm:twh} has the property that $\mathcal{O}(E)/\mathcal{F}$ is torsion free (because $\mathcal{O}(E)/\mathcal{F}$ is locally free on the complement of $S$, which has codimension at least $2$).
\end{remark}

To prove Proposition \ref{prop:destab}, we now suppose that $f_\epsilon$ is such that $\limsup\limits_{\epsilon \rightarrow0} || {\rm log}(f_\epsilon) ||_{L^2} = \infty$. From this we will construct a transverse weakly holomorphic subbundle such that the resulting transverse coherent sheaf violates stability. Following \cite[Section 3.4]{lt}, for $\epsilon >0$ and $x \in X$, let $\lambda(\epsilon , x)$ be the largest eigenvalue of $log( f_\epsilon(x))$ and define
\[
M_\epsilon = {\rm max}_{x \in X} \lambda(\epsilon , x), \quad \rho(\epsilon) = e^{-M_\epsilon}.
\]

The following lemma is proved in the same way as in the non-foliated setting \cite{lt}:
\begin{lemma}[\cite{lt}, Lemma 3.4.5]\label{lem:345}
\leavevmode
\begin{itemize}
\item[(i)]{For every $x \in X$, every eigenvalue of $\rho(\epsilon)f_\epsilon(x)$ satisfies $\lambda \le 1$.}
\item[(ii)]{For every $x \in X$, there is an eigenvalue of $\rho(\epsilon)f_\epsilon(x)$ such that $\lambda \le \rho(\epsilon)$.}
\item[(iii)]{We have ${\rm max}_X ( \rho(\epsilon)|f_\epsilon|) \ge 1$.}
\item[(iv)]{For a subsequence $\epsilon_i \to 0$, we have $\rho(\epsilon_i) \to 0$.}
\end{itemize}
\end{lemma}

Using this lemma, we are at last able to construct the desired transverse weakly holomorphic subbundle. Again, the proof is essentially the same as in \cite{lt}:
\begin{proposition}[\cite{lt}, Proposition 3.4.6]\label{prop:346}
There is a subsequence $\epsilon_i \to 0$ such that $\rho(\epsilon_i) \to 0$ and such that the sequence $f_i = \rho(\epsilon_i)f_{\epsilon_i}$ satisfies:
\begin{itemize}
\item[(i)]{As $i \to \infty$, the $f_i$ converge weakly in $L^{2,1}_B$ to some $f_\infty \neq 0$.}
\item[(ii)]{There is a sequence $\sigma_j$ with $0 < \sigma_j \le 1$, $\sigma_j \to 0$, such that $f_\infty^{\sigma_j}$ converges weakly in $L^{2,1}_B$ to some $f_\infty^0$.}
\item[(iii)]{We have that $\pi = (Id_E - f^0_\infty)$ is a weakly holomorphic subbundle of $E$.}
\end{itemize}
\end{proposition}

Let $\pi$ be the transverse weakly holomorphic subbundle produced by this proposition. Let $\mathcal{F}$ be the corresponding transverse coherent sheaf given in Theorem \ref{thm:twh}. Then as in \cite[Corollary 3.4.7]{lt}, we have that $\mathcal{F}$ is a proper subsheaf of $\mathcal{O}(E)$, that is
\[
0 < {\rm rk}(\mathcal{F}) < r = {\rm rk}(E).
\]
To complete the proof, it remains only to show that $\mu(\mathcal{F}) \ge \mu(E)$, contradicting stability of $E$. The key result in showing this is Proposition \ref{prop:cw}, which says in the foliated setting that the degree of a transverse coherent subsheaf can be calculated using the Chern-Weil formula. In more detail, let $X,S$ and $\mathcal{F}$ be as in Theorem \ref{thm:twh} and let $s = {\rm rk}(\mathcal{F})$. On $X \setminus S$, $\mathcal{F}$ is the sheaf of basic holomorphic sections of a transverse holomorphic subbundle $F \subseteq E|_{X \setminus S}$. Let $h_1$ be the transverse Hermitian metric on $F$ obtained by restriction to $F$ of the transverse Hermitian metric on $E$. Let $F_1$ denote the curvature $2$-form of the Chern connection on $F$ associated to $h_1$. Then Proposition \ref{prop:cw} says that $F_1$ is in $L^1_B$ on $X \setminus S$ and
\begin{equation}\label{equ:cw1}
{\rm deg}(\mathcal{F}) = \frac{i}{2\pi} \int_{X \setminus S} tr( F_1 ) \wedge \omega^{n-1} \wedge \chi.
\end{equation}
With Equation (\ref{equ:cw1}) at hand, the rest of the proof that $\mu(\mathcal{F}) \ge \mu(E)$ is essentially the same as \cite[Pages 88-90]{lt}. This completes the proof of the transverse Hitchin-Kobayashi correspondence.

\section{Applications}\label{sec:appl}

\subsection{A transverse Narasimhan-Seshadri theorem}\label{sec:trns}

In this section, we consider the special features of the transverse Hitchin-Kobayashi correspondence in the special case where $n=1$, that is, the case that the foliation has complex codimension $1$. We obtain a foliated analogue of the Narasimhan-Seshadri theorem \cite{ns}. We also prove existence and uniqueness of the analogue of the Harder-Narasimhan filtration for transverse holomorphic bundles.\\

Throughout this section we let $X$ be a compact oriented, taut, transverse Hermitian foliated manifold of complex codimension $n=1$ and let $g$ be the transverse Hermitian metric on $X$. Since $n=1$ the metric $g$ is not only transverse Gauduchon, but in fact transverse K\"ahler meaning that the transverse Hermitian form $\omega$ satisfies $d\omega = 0$. 

\begin{remark}\label{rem:n=1}
We note here some special features of the $n=1$ case:
\begin{itemize}
\item[(i)]{Let $E$ be a transverse complex vector bundle which admits a transverse Hermitian metric and a transverse unitary connection $\nabla$ with curvature $F_\nabla$. Then $F_\nabla$ is automatically of type $(1,1)$, so $\overline{\partial}_E = \nabla^{0,1}$ is an integrable $\overline{\partial}$-connection and $E$ inherits a transverse holomorphic structure.}
\item[(ii)]{Let $E$ be a transverse holomorphic vector bundle equipped with transverse Hermitian metric and let $\nabla$ be the Chern connection with curvature $F_\nabla$. The Hermitian-Einstein equation $i \Lambda F_\nabla = \gamma id_E$ is equivalent to $F_\nabla = -i\gamma \omega \otimes Id_E$, or by Proposition \ref{prop:einsteinfactor} to:
\[
F_\nabla = -2 \pi i \frac{\mu(E)}{Vol(X)} \omega \otimes Id_E,
\]
where $Vol(X) = \int_X \omega \wedge \chi$. In particular, this says that the connection $\nabla$ is {\em projectively flat}. If the degree of $E$ is zero, the Hermitian-Einstein equations corresponds to $\nabla$ being a flat connection.}
\item[(iii)]{Let $\mathcal{F}$ be a torsion free transverse coherent sheaf. Then $\mathcal{F} = \mathcal{O}(F)$ is the sheaf of basic holomorphic sections of a transverse holomorphic vector bundle $F$. This is because a torsion free coherent sheaf is locally free outside a transverse analytic subset $S \subset X$ of complex codimension $2$. But if $X$ has complex codimension $1$, then $S$ must be empty.}
\item[(iv)]{In particular, if $E$ is a transverse holomorphic vector bundle and $\mathcal{F} \subseteq \mathcal{O}(E)$ is a transverse coherent subsheaf with torsion free quotient, then $\mathcal{F} = \mathcal{O}(F)$, for a transverse holomorphic subbundle $F \subset E$.} 
\item[(v)]{Let $E$ be a transverse holomorphic vector bundle which admits a transverse Hermitian metric. By (iv), we have that stability (resp. semistability) of $E$ is equivalent to: for every proper, non-trivial transverse holomorphic subbundle $F \subset E$, we have
\[
\mu(F) < \mu(E) \quad ({\rm resp.} \; \; \mu(F) \le \mu(E) ).
\]
Note that $\deg(F)$ is well-defined because we can equip $F$ with the induced transverse Hermitian metric. As usual $E$ is polystable if it is the direct sum of stable bundles of the same slope.}
\end{itemize}
\end{remark}

From Remark \ref{rem:n=1} (ii), the Hitchin-Kobayashi correspondence in the case $n=1$ becomes:

\begin{theorem}[Transverse Narasimhan-Seshadri theorem]\label{thm:transns}
Let $X$ be a compact oriented, taut, transverse Hermitian foliated of complex codimension $n=1$. Let $E$ be a transverse holomorphic vector bundle which admits transverse Hermitian metrics. Then $E$ admits a transverse Hermitian metric such that the Chern connection $\nabla$ satisfies
\[
F_\nabla = - 2\pi i \frac{\mu(E)}{Vol(X)} \omega \otimes Id_E,
\]
if and only if $E$ is polystable.
\end{theorem}

\begin{remark}
It is tempting to try and give a more direct proof of Theorem \ref{thm:transns} by adapting Donaldson's proof of the Narasimhan-Seshadri theorem \cite{don} to the foliated setting. However that proof uses Uhlenbeck compactness which fails in the foliated setting (see Example \ref{ex:uhlenbeckcomp} below), so the proof can not be easily adapted.
\end{remark}

Of particular interest is the case of transverse holomorphic bundles of degree $0$. By the transverse Narasimhan-Seshadri theorem, if $E$ is a rank $m$ degree $0$ polystable transverse holomorphic vector bundle admitting transverse Hermitian metrics, then $E$ admits a transverse Hermitian metric such that the Chern connection $\nabla$ is flat. Since $\nabla$ is a flat connection it is given by a unitary representation of the fundamental group $\rho : \pi_1(X) \to U(m)$. Conversely, given such a representation $\rho : \pi_1(X) \to U(m)$ we obtain a complex rank $m$ Hermitian vector bundle $E$ equipped with a flat unitary connection $\nabla$. Let $P_E \to X$ be the associated principal $U(m)$-bundle. Since $\nabla$ is flat, we can use it to lift the foliation on $X$ to a foliation on $P$. In this way $E$ inherits the structure of a foliated vector bundle and $\nabla$ becomes a basic connection with respect to this structure. From Remark \ref{rem:n=1} (i), we see that $E$ inherits a transverse holomorphic structure by taking $\overline{\partial}_E = \nabla^{0,1}$. Our transverse Narasimhan-Seshadri correspondence can thus be summarised as:

\[
\left\{ \begin{array}{c} \text{Isomorphism classes of} \\ \text{rank } m \text{ unitary} \\ \text{ representations of } \pi_1(X) \end{array} \right\} \leftrightarrow
\left\{ \begin{array}{c} \text{Isomorphism classes of polystable rank } m \\ \text{degree } 0 \text{ transverse holomorphic vector bundles} \\ \text{admitting transverse Hermitian metrics} \end{array} \right\}
\]
\\

Moreover, by Corollary \ref{cor:irredsimple}, we see that $E$ is stable if and only if the representation $\rho$ is irreducible.\\

To each transverse holomorphic vector bundle $E$, there is an underlying foliated complex vector bundle obtained by forgetting the holomorphic structure but remembering the transverse structure. We emphasise that in the above correspondence, one has to consider {\em all} possible transverse structures on $E$. If the underlying transverse structure on $E$ is kept fixed, then we only only obtain a subset of unitary representations $\rho : \pi_1(X) \to U(m)$. The following example illustrates this phenomenon and moreover shows that the analogue of the Uhlenbeck compactness theorem \cite{uhl} does not hold in general for basic connections on a vector bundle with fixed transverse structure.

\begin{example}\label{ex:uhlenbeckcomp}
Let $X = T^3 = \mathbb{R}^3/\mathbb{Z}^3$ be the standard $3$-torus and let $(x^1,x^2,x^3)$ denote the standard coordinates on $\mathbb{R}^3$. We equip $X$ with the $1$-dimensional foliation $\mathcal{F} = \langle \xi \rangle$, where
\[
\xi = \xi^1 \frac{\partial}{\partial x^1} + \xi^2 \frac{\partial}{\partial x^2} + \xi^3 \frac{\partial}{\partial x^3},
\]
where we assume that $\xi^1 , \xi^2 , \xi^3$ are rationally independent. The Euclidean metric on $\mathbb{R}^3$ descends to a metric on $X$ and we note that $X$ has a naturally defined transverse K\"ahler structure. Flat unitary line bundles on $X$ are classified by their holonomy homomorphism $\rho : H_1(X , \mathbb{Z}) \to U(1)$. Let $\gamma^1 , \gamma^2 , \gamma^3$ be the cycles corresponding to the three circle factors comprising $T^3 = (S^1)^3$. Then $\rho$ corresponds to an element of the dual torus $\hat{T}^3 = U(1)^3 = (\mathbb{R}^3)^*/(\mathbb{Z}^3)^*$ via $\rho \mapsto ( \rho(\gamma^1) , \rho(\gamma^2) , \rho(\gamma^3) )$. Let us regard $(\mathbb{R}^3)^*$ as the universal cover of the dual torus $\hat{T}^3$ and let $q : (\mathbb{R}^3)^* \to \hat{T}^3$ be the covering map. A point $y = (y_1 , y_2 , y_3) \in (\mathbb{R}^3)^*$ determines a flat unitary connection on the trivial line bundle $\mathbb{C} \times T^3$, namely $\nabla = d + i\alpha_y$, where
\[
\alpha_y = y_1 dx^1 + y_2dx^2 + y_3 dx^3.
\]
Every flat unitary connection is up to gauge equivalence of this form for some $y \in (\mathbb{R}^3)^*$. Clearly $y = (y_1 , y_2 , y_3)$ and $\hat{y} = (\hat{y}_1 , \hat{y}_2 , \hat{y}_3)$ define gauge equivalent connections if and only if $\hat{y} - y \in \mathbb{Z}^3$, so that the gauge equivalence class of $y$ corresponds to the point $q(y) \in \hat{T}^3$ in the dual torus. The flat connections $d + i\alpha_y , d + i\alpha_{\hat{y}}$ corresponding to a pair of points $y,\hat{y} \in (\mathbb{R}^3)^*$ induce the same transverse structure on the trivial line bundle if and only if $\alpha_{\hat{y}} - \alpha_y = (\hat{y}_1 - y_1)dx^1 + (\hat{y}_2 - y_2)dx^2 + (\hat{y}_3 - y_3)dx^3$ is a basic $1$-form, or equivalently if and only if
\[
(\hat{y}_1 - y_1)\xi^1 + (\hat{y}_2 - y_2)\xi^2 + (\hat{y}_3 - y_3)\xi^3 = 0.
\]
In particular $y$ and $\hat{y}$ define flat connections which are gauge equivalent by a {\em basic gauge transformation} if and only if $( \hat{y}_1 - y_1 , \hat{y}_2 - y_2 , \hat{y}_3 - y_3) = (m_1 , m_2 , m_3 ) \in \mathbb{Z}^3$ satisfies $m_1 \xi^1 + m_2 \xi^2 + m_3 \xi^3 = 0$. As we assume $\xi^1 , \xi^2 , \xi^3 $ are rationally independent, the only solutions of this equation are $m_1 = m_2 = m _3 = 0$, i.e. $y = \hat{y}$. We deduce the following:
\begin{itemize}
\item{The $2$-planes $y_1 \xi^1 + y_2 \xi^2 + y_3 \xi^3 = const$ determine a foliation of $(\mathbb{R}^3)^*$ which descends via $q : (\mathbb{R}^3)^* \to \hat{T}^3$ to an irregular foliation on $\hat{T}^3$.}
\item{The leaves of the induced foliation on $\hat{T}^3$ correspond to flat unitary line bundles on $X$ with fixed transverse structure.}
\item{Each leaf is an immersed, but not embedded copy of the non-compact manifold $\mathbb{R}^2$ inside $\hat{T}^3$.}
\end{itemize}
The non-compactness of the leaves of the foliation on $\hat{T}^3$ signals the failure of Uhlenbeck compactness to hold for basic unitary connections on a unitary vector bundle with {\em fixed transverse structure}. Indeed if Uhlenbeck compactness were true in this sense, it would follow that the moduli space of flat basic unitary connections on a vector bundle with fixed transverse structure over a compact space $X$ would be compact. We have just shown an example where the moduli space of flat unitary line bundles for fixed transverse structure is $\mathbb{R}^2$, which in particular is non-compact.
\end{example}

Aside from the Narasimhan-Seshadri theorem, a natural result to extend to the foliated setting is the Harder-Narasimhan filtration:
\begin{theorem}[Transverse Harder-Narasimhan filtration]\label{thm:hn}
Let $X$ be a compact oriented, taut, transverse Hermitian foliated manifold of complex codimension $n=1$. Let $E$ be a transverse holomorphic vector bundle which admits transverse Hermitian metrics. There exists a uniquely determined filtration of $E$
\[
0 = E_0 \subset E_1 \subset \cdots \subset E_k = E
\]
by transverse holomorphic subbundles such that the quotients $F_i = E_i/E_{i-1}$ are semistable and the slopes are strictly increasing:
\[
\mu(F_1) > \mu(F_2) > \cdots > \mu(F_k).
\]
We call this the {\em (transverse) Harder-Narasimhan filtration} of $E$.
\end{theorem}
In order to prove Theorem \ref{thm:hn}, we need some preliminary lemmas. Suppose that $F \subset E$ is a proper non-trivial subbundle. Let $\pi \in \mathcal{C}^\infty_B( X  , End(E) )$ be the orthogonal projection from $E$ to $F$ and let $F_1$ denote the curvature of the Chern connection on $F$ induced by the inclusion $F \to E$. Then as in the proof of Theorem \ref{thm:heispolystable}, we have:
\[
i \Lambda tr( F_1 ) =  i\Lambda tr( \pi F_E \pi ) - | \partial_{End(E)} \pi |^2.
\]
Wedging with $\omega \wedge \chi$ and integrating over $X$, we obtain:
\begin{equation}\label{equ:dpiL2}
i \int_X tr(F_1) \wedge \chi =  \int_X tr( \pi i \Lambda(F_E) \pi ) \wedge \omega \wedge \chi -  || \partial_{End(E)} \pi ||^2_{L^2}.
\end{equation}
This shows that the set $\{ \mu(F) \; | \; F \subset E \}$ of slopes of all possible proper non-trivial transverse holomorphic subbundles $F \subset E$ is bounded above. Let $\hat{\mu}$ be the supremum. We then have:
\begin{lemma}\label{lem:max}
There exists a proper non-trivial transverse holomorphic subbundle $F \subset E$ such that $\mu(F) = \hat{\mu}$.
\end{lemma}
\begin{proof}
Let $\{ F_i \}$ be a sequence of proper non-trivial transverse holomorphic subbundles of $E$ such that $\mu(F_i) \to \hat{\mu}$ as $i \to \infty$. Since there are only finitely many possible ranks that a subbundle of $E$ can have, we may as well assume that the $F_i$ all have the same rank, say $rk(F_i) = s$. Let $\pi_i \in \mathcal{C}^\infty_B( X  , End(E) )$ be the orthogonal projection from $E$ to $F_i$. We can view the $\pi_i$ as being transverse weakly holomorphic subbundles. The idea now is prove that after passing to a subsequence, we can obtain $L^{2,1}_B$-convergence of the $\pi_i$ to a limiting transverse weakly holomorphic subbundle $\pi$. We have already seen that a transverse weakly holomorphic subbundle is given by a genuine transverse holomorphic subbundle outside of a transverse codimension $2$ subset $S \subset X$. But since $X$ has codimension $1$, this means $\pi$ is everywhere given by a transverse holomorphic subbundle. Then we just have to check that the limiting subbundle has slope $\hat{\mu}$.\\

First we look for a uniform $L^{2,1}_B$ bound on the $\{ \pi_i \}$. Since the $\pi_i$ are projections, they satisfy $\pi_i = \pi^*_i = \pi^2_i$, which implies a uniform $L^2_B$ bound. Next by (\ref{equ:dpiL2}) applied to the subbundles $F_i$, we obtain
\[
i \int_X tr(F_{1,i}) \wedge \chi =  \int_X tr( \pi_i i \Lambda(F_E) \pi_i ) \wedge \omega \wedge \chi -  || \partial_{End(E)} \pi_i ||^2_{L^2},
\]
where $F_{1,i}$ is the curvature of the induced Chern connection on $F_i$. Then since $\deg(F_i) \to s \hat{\mu}$ as $i \to \infty$, we have a uniform bound on $i \int_X tr(F_{1,i}) \wedge \chi$. Clearly we can also bound $\int_X tr( \pi_i i \Lambda(F_E) \pi_i ) \wedge \omega \wedge \chi$ uniformly and so this gives a uniform bound for the $L^2$-norm of $\partial_{End(E)} \pi_i$ and therefore we obtain a uniform $L^{2,1}_B$-bound on the $\pi_i$. Now we apply the Banach-Alaogu theorem so that on passing to a subsequence, the $\pi_i$ converge weakly in $L^{2,1}_B$ to some $\pi \in L^{2,1}_B( X , End(E) )$. We claim that $\pi$ is a transverse weakly holomorphic subbundle. By Sobolev compactness, weak convergence in $L^{2,1}_B$ implies strong convergence in $L^2_B$ and it follows that $\pi = \pi^* = \pi^2$ in $L^1_B$. It remains to show that $(1-\pi)\overline{\partial}_E \pi = 0$ in $L^1_B$. Using $\pi^2 = \pi$, this is equivalent to showing
\[
(\overline{\partial}_E \pi) \pi = 0 \; \; \text{in } L^1_B.
\]
Note that weak convergence $\pi_i \to \pi$ in $L^{2,1}_B$ implies that
\begin{equation}\label{equ:conv}
\overline{\partial}_E \pi_i \to \overline{\partial}_E \pi \; \; \; \text{weakly in } L^2_B. 
\end{equation}
Also, since $\pi_i$ is the projection to the transverse holomorphic subbundle $F_i$, we have
\begin{equation*}
( \overline{\partial}_E \pi_i) \pi_i = 0 \; \; \text{for all } i.
\end{equation*}
We then have
\begin{equation*}
\begin{aligned}
|| (\overline{\partial}_E \pi) \pi ||_{L^1_B} &= || (\overline{\partial}_E \pi)\pi - (\overline{\partial}_E \pi_i)\pi_i ||_{L^1_B} \\
&= || (\overline{\partial}_E \pi)\pi - (\overline{\partial}_E \pi_i )\pi + (\overline{\partial}_E \pi_i )\pi - (\overline{\partial}_E \pi_i)\pi_i ||_{L^1_B}  \\
&\le || ( \overline{\partial}_E \pi - \overline{\partial} \pi_i ) \pi ||_{L^1_B} + || \overline{\partial}_E \pi_i ( \pi - \pi_i) ||_{L^1_B}.
\end{aligned}
\end{equation*}
But $|| ( \overline{\partial}_E \pi - \overline{\partial} \pi_i ) \pi ||_{L^1_B} \to 0$ by (\ref{equ:conv}) and
\[
|| \overline{\partial}_E \pi_i ( \pi - \pi_i) ||_{L^1_B} \le || \overline{\partial}_E \pi_i ||_{L^2_B} || \pi - \pi_i ||_{L^2_B} \to 0
\]
since $|| \overline{\partial}_E \pi_i ||_{L^2_B}$ is bounded and $\pi_i \to \pi$ strongly in $L^2_B$. This shows that $(\overline{\partial}_E \pi) \pi = 0$ in $L^1_B$, as required. Therefore $\pi$ is a transverse weakly holomorphic subbundle, and as argued above, $\pi$ must actually arise from a genuine holomorphic subbundle $F \subset E$, since $n = 1$. Clearly $rk(F) = s$, so to show that $\mu(F) = \hat{\mu}$, it suffices to show that $\deg(F) = \lim_{i \to \infty} \deg(F_i)$. From (\ref{equ:dpiL2}) applied to $\pi_i$ and to $\pi$, we get
\begin{equation}\label{equ:integral1}
2 \pi \deg(F_i) + || \partial_{End(E)} \pi_i ||_{L^2_B}^2 = \int_X tr( \pi_i i \Lambda(F_E) \pi_i ) \wedge \omega \wedge \chi
\end{equation}
and
\begin{equation}\label{equ:integral2}
2 \pi \deg(F) + || \partial_{End(E)} \pi ||_{L^2_B}^2 = \int_X tr( \pi i \Lambda(F_E) \pi ) \wedge \omega \wedge \chi.
\end{equation}
But since $\pi_i \to \pi$ weakly in $L^{2,1}_B$, we have $|| \partial_{End(E)} \pi ||_{L^2_B}^2 \le \liminf_{i \to \infty} || \partial_{End(E)} \pi_i ||_{L^2_B}^2$ and
\[
\int_X tr( \pi_i i \Lambda(F_E) \pi_i ) \wedge \omega \wedge \chi \to \int_X tr( \pi i \Lambda(F_E) \pi ) \wedge \omega \wedge \chi \; \; \text{as } i \to \infty.
\]
Using (\ref{equ:integral1}) and (\ref{equ:integral2}), we therefore have
\begin{equation*}
\begin{aligned}
2 \pi \deg(F) + || \partial_{End(E)} \pi ||_{L^2_B}^2 &= \int_X tr( \pi i \Lambda(F_E) \pi ) \wedge \omega \wedge \chi \\
&= \liminf_{i \to \infty} \int_X tr( \pi_i i \Lambda(F_E) \pi_i ) \wedge \omega \wedge \chi \\
&= \liminf_{i \to \infty} ( 2\pi \deg(F_i) + ||\partial_{End(E)} \pi_i ||_{L^2_B}^2 ) \\
&\ge 2\pi s \hat{\mu} + || \partial_{End(E)} \pi ||_{L^2_B}^2.
\end{aligned}
\end{equation*}
Thus $\deg(F) \ge s \hat{\mu}$ and so $\mu(F) \ge \hat{\mu}$. But from the definition of $\hat{\mu}$ we must also have $\mu(F) \le \hat{\mu}$ and hence $\mu(F) = \hat{\mu}$.
\end{proof}

\begin{lemma}[Jordan-H\"older filtration]\label{lem:jh}
Let $X$ be a compact oriented, taut, transverse Hermitian foliated manifold of complex codimension $n=1$. Let $E$ be a transverse holomorphic bundle which admits transverse Hermitian metrics. If $E$ is semistable, then $E$ admits a filtration
\[
0 = E_0 \subset E_1 \subset \cdots \subset E_k = E
\]
by transverse holomorphic subbundles such that the quotients $F_i = E_i/E_{i-1}$ are all stable and satisfy $\mu(F_i) = \mu(E)$.
\end{lemma}
\begin{proof}
The proof is by induction on the rank of $E$ and is straightforward.
\end{proof}
Note that the filtration in the above lemma is in general not unique.

\begin{lemma}\label{lem:vanishing}
Let $X$ be a compact oriented, taut, transverse Hermitian foliated manifold of complex codimension $n=1$. Let $E,F$ be transverse holomorphic bundles admitting transverse Hermitian metrics. Suppose that $E,E'$ are semistable and that $\mu(E) > \mu(E')$. Then any basic holomorphic endomorphism $\alpha : E \to E'$ vanishes.
\end{lemma}
\begin{proof}
Let
\[
0 = E_0 \subset E_1 \subset \cdots \subset E_k = E
\]
and
\[
0 = E'_0 \subset E'_1 \subset \cdots \subset E'_l = E'
\]
be filtrations on $E$ and $E'$ as in Lemma \ref{lem:jh}. We first show that $\alpha|_{E_1} = 0$. To see this consider the projection $\alpha|_{E_1} : E_1 \to E' \to E'/E'_{l-1} = F'_{l}$. We have that $E_1$ and $F'_l$ are stable and that $\mu(E_1) = \mu(E) > \mu(E') = \mu(F'_l)$. Thus $E_1$ and $F'_l$ admit Hermitian-Einstein metrics, inducing a Hermitian-Einstein metric on $Hom(E_1 , F'_l)$. But $\mu( Hom(E_1 , F'_l) ) < 0$, so by Theorem \ref{thm:vanishing}, any basic holomorphic section of $Hom(E_1 , F'_l)$ is zero. This shows that $\alpha|_{E_1}$ maps into $E'_{l-1}$. Continuing in this fashion we eventually get that $\alpha|_{E_1} = 0$. Now replacing $E$ by $E/E_1$ and arguing as above, we see that $\alpha|_{E_2} = 0$. Continuing in this fashion we find that $\alpha = 0$.
\end{proof}

\begin{proof}[Proof of Theorem \ref{thm:hn}]
We first prove existence. If $E$ is semistable, we are done. Otherwise, by Lemma \ref{lem:max} there exists a proper non-trivial transverse holomorphic subbundle $E_1 \subset E$ such that $\mu(E_1)$ is maximal. We also choose $E_1$ to have maximal rank amongst all such subbundles of $E$. By maximality of $\mu(E_1)$, it follows that $E_1$ is semistable. If $E/E_1$ is not semistable, choose a proper non-trivial subbundle $F_2 \subset E/E_1$ with maximal slope and with rank maximal amongst all such subbundles of $E/E_1$. Then $F_2$ is semistable. Let $E_2$ be the preimage of $F_2$ with respect to the projection $E \to E/E_1$. Continuing in this fashion we obtain a filtration 
\[
0 = E_0 \subset E_1 \subset \cdots \subset E_k = E
\]
by transverse holomorphic subbundles such that the quotients $F_i = E_i/E_{i-1}$ are semistable. To see that the slopes are strictly increasing, suppose that $\mu(F_i) \le \mu(F_{i+1})$. From the short exact sequence $0 \to F_i \to E_{i+1}/E_{i-1} \to F_{i+1} \to 0$, we obtain $\mu( E_{i+1}/E_{i-1}) \ge \mu(F_i)$. If $\mu(E_{i+1}/E_{i-1}) > \mu(F_i)$, this contradicts maximality of the slope of $F_i$ as a subbundle of $E/E_{i-1}$. If $\mu(E_{i+1}/E_{i-1}) = \mu(F_i)$, this contradicts maximality of the rank of $F_i$ amongst all subbundles of $E/E_{i-1}$ with maximal slope. Therefore $\mu(F_i) > \mu(F_{i+1})$, so the slopes are strictly increasing.\\

We now prove uniqueness. Let
\[
0 = E'_0 \subset E'_1 \subset \cdots \subset E'_l = E
\]
be another filtration by transverse holomorphic subbundles such that the quotients $F'_i = E'_i/E'_{i-1}$ are semistable and the slopes are strictly increasing. By the definition of $E_1$, we have $\mu(E_1) \ge \mu(E'_1)$ and thus $\mu(E_1) \ge \mu(E'_1) = \mu(F'_1) > \mu(F'_2) > \cdots > \mu(F'_l)$. So by Lemma \ref{lem:vanishing}, the map $\alpha : E_1 \to E \to E/E'_l$ is zero, hence $E_1 \subseteq E'_{l-1}$. Continuing in this manner we eventually see that $E_1 \subseteq E'_1$. Applying Lemma \ref{lem:vanishing} to the inclusion $E_1 \to E'_1$ we see that $\mu(E_1) = \mu(E'_1)$. Then by maximality of the rank of $E_1$ amongst all subbundles of $E$ with maximal slope, we see that $E_1 = E'_1$. Repeating the above argument for $E/E_1$ in place of $E$, we see that $E_2 = E'_2$. Continuing in this way, we get that  the two filtrations of $E$ coincide.
\end{proof}

\subsection{Hitchin-Kobayashi for Sasakian manifolds}\label{sec:sasaki}

Sasakian geometry provides a wealth of interesting examples of taut, transverse Hermitian foliations and were the original motivation for us to develop a transverse Hitchin-Kobayashi correspondence. For these reasons it seems worthwhile to recall the definition of Sasaki manifolds and to explicitly state the transverse Hitchin-Kobayashi correspondence for them, which we do in Corollary \ref{cor:hksasaki}. In Section \ref{sec:instantons}, we will see that the Hitchin-Kobayshi correspondence for Sasaki manifolds is relevant to the study of higher-dimensional instantons, such as contact instantons in $5$-dimensions.\\

Let $X$ be a manifold of dimension $2n+1$. Recall that an {\em almost contact metric structure} $(\xi,\eta,\Phi,g_X)$ on $X$ consists of a vector field $\xi$, $1$-form $\eta$, endomorphism $\Phi \colon TX \to TX$ and a Riemannian metric $g_X$ such that $\eta(\xi) =1$, $\Phi^2 = -I + \eta \otimes \xi$ and $g_X( \Phi U , \Phi V) = g_X(U,V) - \eta(U) \eta(V)$ for all vector fields $U,V$. Equivalently this is a reduction of structure of the tangent bundle to $U(n) \subset GL(2n+1,\mathbb{R})$. We let $V$ be the rank $1$ subbundle spanned by $\xi$ and $H = {\rm Ker}(\eta)$ the annihilator of $\eta$. Then we have an orthogonal decomposition $TX = V \oplus H$ together with a unitary structure on $H$. We think of $\xi$ as generating a $1$-dimensional foliation on $X$. Thus $V$ is the distribution tangent to the foliation and $H$ is its orthogonal complement. Furthermore $\eta$ defines a leafwise volume form. We also let $g$ denote the restriction of $g_X$ to $H$.\\

The restriction $J = \Phi|_H$ of $\Phi$ to $H$ defines a complex structure on $H$ and letting $\omega(U,V) = g(U,\Phi V)$, we have that $\omega$ is a $2$-form which restricted to $H$ is the Hermitian $2$-form associated to $J$. We say that $X$ is a {\em contact metric manifold} if in addition $d\eta = \omega$. This implies that $\eta$ is a contact form and $\xi$ the associated Reeb vector field. Moreover it also implies that the foliation is taut.\\

We recall that a contact metric manifold $X$ called {\em $K$-contact} if $\xi$ is a Killing vector for $g_X$. In this case, we have that $(\xi , g)$ is a taut Riemannian foliation and $J$ is a transverse almost complex structure. Recall that $X$ is called {\em Sasakian} if the transverse almost complex structure $J$ is integrable \cite[\textsection 6]{bg}. In particular, a Sasakian manifold of dimension $2n+1$ is equipped with a taut, transverse Hermitian foliation of complex codimension $n$. Note that $d\eta= \omega$ implies that $\omega$ is closed, so Sasakian manifolds are transverse K\"ahler and in particular transverse Gauduchon. Thus if $X$ is compact, the transverse Hitchin-Kobayshi correspondence applies:
\begin{corollary}\label{cor:hksasaki}
Let $X$ be a compact Sasakian manifold. Then the Hitchin-Kobayshi correspondence holds for $X$. Namely, a transverse holomorphic vector bundle $E$ which admits transverse Hermitian metrics admits a transverse Hermitian-Einstein metric if and only if $E$ is polystable.
\end{corollary}

\begin{remark}
Corollary \ref{cor:hksasaki} was proven in the special case of compact quasi-regular Sasaki manifolds in \cite{bisc}.
\end{remark}

\begin{remark}
The only property of Sasakian manifolds we have used for this result is that they are transverally K\"ahler. Other well-known examples of transversally K\"ahler geometries include $3$-Sasakian manifolds \cite{bg} and co-K\"ahler manifolds \cite{li}.
\end{remark}

\subsection{Contact instantons and higher-dimensional generalisations}\label{sec:instantons}

In this section we look at the relation between the transverse Hermitian-Einstein equations and various types of higher-dimensional instantons. We find a number of instances of higher-dimensional instanton equations which are special cases of the transverse Hermitian-Einstein equations.\\

The usual anti-self-dual instanton equation for a connection $A$ on a $4$-manifold are given by $*F_A = -F_A$. There is a natural extension of the anti-self-duality equations to $d \ge 4$ dimensions given by choosing a $(d-4)$-form $\Omega$. We say that a connection $A$ is an {\em $\Omega$-instanton} \cite{cdfn,dt,tian} if the curvature $2$-form $F_A$ satisfies:
\begin{equation}\label{equ:higherinstanton}
*F_A = -  \Omega \wedge F_A.
\end{equation}
If $\Omega$ is closed, then differentiating in (\ref{equ:higherinstanton}) and using the Bianchi identity, one finds that $A$ satisfies the Yang-Mills equations $d_A( *F_A) = 0$. However, we will see that there are examples of $(d-4)$-forms $\Omega$ which are not closed and yet every solution of (\ref{equ:higherinstanton}) satisfies the Yang-Mills equations.\\

We will say that a connection $A$ on a vector bundle $E$ has {\em trivial determinant} if the induced connection on $\det(E)$ admits a constant section, that is, if $E$ admits covariantly constant volume form. Clearly this implies that the curvature of $A$ is trace-free.

\begin{proposition}\label{prop:heinstanton}
Let $X$ be an oriented, taut, transverse Hermitian foliation of complex codimension $n$. Let $E$ be a foliated complex vector bundle with a basic Hermitian metric and let $A$ be a basic unitary connection with trivial determinant with curvature $F_A$. Then $A$ is a transverse Hermitian-Einstein connection if and only if $A$ is an $\Omega$-instanton, where $\Omega = \frac{\omega^{n-2}}{(n-2)!} \wedge \chi$.
\end{proposition}
\begin{proof}
This is proven in the non-foliated setting in \cite{tian}. The result clearly generalises to the foliated case by replacing $\frac{\omega^{n-2}}{(n-2)!}$ with $\Omega = \frac{\omega^{n-2}}{(n-2)!} \wedge \chi$.
\end{proof}

\begin{remark}
Suppose that $X$ is compact and that the transverse metric is Gauduchon. If $E$ is a polystable transverse holomorphic bundle and $\det(E) = \mathcal{O}$ is trivial as a transverse holomorphic line bundle, then we claim that the associated Hermitian-Einstein connection $A$ has trivial determinant and is therefore an $\Omega$-instanton for $\Omega = \frac{\omega^{n-2}}{(n-2)!} \wedge \chi$, by Proposition \ref{prop:heinstanton}. To see this, note that the Hermitian-Einstein connection on $E$ induces a Hermitian-Einstein connection on $\det(E) = \mathcal{O}$. By Proposition \ref{prop:unique}, the induced Hermitian-Einstein connection on $\det(E)$ is unique, so must be the trivial flat connection.
\end{remark}

\begin{proposition}
Let $X$ be an oriented, taut, transverse Hermitian foliation of complex codimension $n$. Suppose that the leafwise volume form $\chi$ satisfies $d\chi = \omega \wedge \theta$ for some $\theta$. Let $\Omega = \frac{\omega^{n-2}}{(n-2)!} \wedge \chi$. Then every $\Omega$-instanton $A$ with trivial determinant is a solution of the Yang-Mills equation $d_A (*F_A) = 0$.
\end{proposition}
\begin{proof}
Differentiating the instanton equation $*F_A = -\Omega \wedge F_A$ and using the Bianchi identity, we find
\begin{equation*}
\begin{aligned}
d_A(*F_A) = - d\Omega \wedge F_A = - \frac{\omega^{n-2}}{(n-2)!} \wedge d\chi \wedge F_A &= -\frac{\omega^{n-1}}{(n-2)!} \wedge \theta \wedge F_A \\
&= - \theta \wedge \frac{\omega^n}{(n-1)!} (\Lambda F_A).
\end{aligned}
\end{equation*}
Since $A$ is an $\Omega$-instanton we have that $A$ is also a Hermitian-Einstein connection, by Proposition \ref{prop:heinstanton}. But $\det(E)$ is trivial, so $\deg(E) = 0$ and the Hermitian-Einstein equations become $\Lambda F_A = 0$. Therefore $d_A(*F_A) = 0$.
\end{proof}

\begin{remark}
The above proposition holds for instance when $X$ is Sasakian, as $\chi = \eta$ and $d\eta = \omega$. Similarly the proposition holds if $X$ is co-K\"ahler, as $\chi = \eta$ and $d\eta =0$.
\end{remark}

Next, we recall the notion of {\em contact instantons} introduced in \cite{kaza} in the context of $5$-dimensional supersymmetric Yang-Mills theory and are a $5$-dimensional analogue of self-dual/anti-self-dual instantons on $4$-manifolds. We studied the moduli space of contact instantons in \cite{bh}. Suppose that $X$ is a $K$-contact $5$-manifold. A connection $A$ on $X$ is called a self-dual contact instanton if it satisfies
\[
*F_A = \eta \wedge F_A.
\]
Similarly $A$ is called an anti-self-dual contact instanton if
\[
*F_A = -\eta \wedge F_A.
\]
Clearly anti-self-dual/self-dual contact instantons are $\Omega$-instantons for $\Omega = \pm \eta$. Of particular interest is the case when $X$ is a Sasakian $5$-manifold. Then by Proposition \ref{prop:heinstanton} we have:
\begin{corollary}\label{cor:continst}
Anti-self-dual $SU(r)$ contact instantons on a $5$-dimensional Sasakian manifold $X$ are precisely the $SU(r)$ $\Omega$-instantons for $\Omega = \eta$, therefore they correspond to rank $r$ transverse Hermitian-Einstein connections with trivial determinant. If $X$ is compact, then by the transverse Hitchin-Kobayashi correspondence, anti-self-dual $SU(r)$ contact instantons on $X$ correspond to rank $r$ polystable transverse holomorphic bundles with trivial determinant.
\end{corollary}

A $7$-dimensional analogue of contact instantons was considered in \cite{prz}, which unsurprisingly arise from the study $7$-dimensional supersymmetric Yang-Mills theory. Let $X$ be a $7$-dimensional $K$-contact manifold. We say that a connection $A$ on $X$ is a self-dual/anti-self-dual {\em higher contact instanton} on $A$ if:
\[
*F_A = \pm \eta \wedge d\eta \wedge F_A.
\]
Naturally, one can generalise this to the case that $X$ is a $K$-contact manifold of dimension $2n+1$ with $n \ge 2$ and define self-dual/anti-self-dual higher contact instantons to be solutions of
\[
*F_A = \pm \eta \wedge \frac{ \; \; \; \; (d\eta)^{n-2}}{(n-2)!} \wedge F_A.
\]
In the case that $X$ is Sasakian, the analogue of Corollary \ref{cor:continst} holds:
\begin{corollary}
Anti-self-dual $SU(r)$ higher contact instantons on a $2n+1$-dimensional Sasakian manifold $X$ (with $n \ge 2)$ are precisely the $SU(r)$ $\Omega$-instantons for $\Omega = \eta \wedge \frac{ \; \; \; \; (d\eta)^{n-2}}{(n-2)!}$, therefore they correspond to rank $r$ transverse Hermitian-Einstein connections with trivial determinant. If $X$ is compact, then by the transverse Hitchin-Kobayashi correspondence, anti-self-dual $SU(r)$ higher contact instantons on $X$ correspond to rank $r$ polystable transverse holomorphic bundles with trivial determinant.
\end{corollary}

\bibliographystyle{amsplain}

\begin{thebibliography}{99}
\bibitem{adams}R. A. Adams, Sobolev spaces. Pure and Applied Mathematics,  {\bf 65}. Academic Press, New York-London, 1975. xviii+268 pp. 
\bibitem{bh}D. Baraglia, P. Hekmati, Moduli spaces of contact instantons, {\em Adv. Math.} {\bf 294} (2016), 562-595. 
\bibitem{bisc}I. Biswas, G. Schumacher. Vector bundles on Sasakian manifolds, {\em Adv. Theor. Math. Phys.} {\bf 14} (2010), no. 2, 541-561. 
\bibitem{bog}F. A. Bogomolov, Holomorphic tensors and vector bundles on projective manifolds, {\em Izv. Akad. Nauk SSSR Ser. Mat.} {\bf 42} (1978), no. 6, 1227-1287, 1439. 
\bibitem{bg}C. P. Boyer, K. Galicki, Sasakian geometry, Oxford University Press, Oxford, (2008), 613 pp.
\bibitem{cdfn}E. Corrigan, C. Devchand, D. B. Fairlie, J. Nuyts, First-order equations for gauge fields in spaces of dimension greater than four, {\em Nuclear Phys. B} {\bf 214} (1983), no. 3, 452-464. 
\bibitem{don}S. K. Donaldson, A new proof of a theorem of Narasimhan and Seshadri, {\em J. Diff. Geom.} {\bf 18} (1983) 269-277.
\bibitem{dt}S. K. Donaldson, R. P. Thomas, Gauge theory in higher dimensions. The geometric universe (Oxford, 1996), 31-47, Oxford Univ. Press, Oxford, (1998).
\bibitem{elk}A. El Kacimi-Alaoui, Op\'erateurs transversalement elliptiques sur un feuilletage riemannien et applications, {\em Compositio Math.} {\bf 73} (1990), no. 1, 57-106.
\bibitem{gri}P. A. Griffiths, Hermitian differential geometry, Chern classes, and positive vector bundles. 1969 Global Analysis (Papers in Honor of K. Kodaira) pp. 185-251 Univ. Tokyo Press, Tokyo.
\bibitem{hir}H. Hironaka, Resolution of singularities of an algebraic variety over a field of characteristic zero. I, II, {\em Ann. of Math.} (2) {\bf 79} (1964), 109-203; ibid. (2) {\bf 79} (1964) 205-326. 
\bibitem{kaza}J. K\"all\'en, M. Zabzine, Twisted supersymmetric 5D Yang-Mills theory and contact geometry, {\em J. High Energy Phys.} {\bf 125} (2012), 25 pp.
\bibitem{kato}F. W. Kamber, P. Tondeur, Foliated bundles and characteristic classes. Lecture Notes in Mathematics, {\bf 493}. Springer-Verlag, Berlin-New York, 1975. xiv+208 pp. 
\bibitem{kob}S. Kobayashi, Differential geometry of complex vector bundles. Publications of the Mathematical Society of Japan, 15. Princeton University Press, Princeton, NJ, 1987. xii+305 pp.
\bibitem{klw}Y. Kordyukov, M. Lejmi, P. Weber, Seiberg-Witten invariants on manifolds with Riemannian foliations of codimension 4, {\em J. Geom. Phys.} {\bf 107} (2016), 114-135. 
\bibitem{li}H. Li, Topology of co-symplectic/co-K\"ahler manifolds, {\em Asian J. Math.} {\bf 12} (2008), no. 4, 527-543. 
\bibitem{lt}M. L\"ubke, A. Teleman, The Kobayashi-Hitchin correspondence. World Scientific Publishing Co., Inc., River Edge, NJ, 1995. x+254 pp.
\bibitem{lt2}M. L\"ubke, A. Teleman, The universal Kobayashi-Hitchin correspondence on Hermitian manifolds. Mem. Amer. Math. Soc. {\bf 183} (2006), no. 863, vi+97 pp. 
\bibitem{mol}P. Molino, Riemannian foliations. Translated from the French by Grant Cairns. With appendices by Cairns, Y. Carri\`ere, \'E. Ghys, E. Salem and V. Sergiescu. Progress in Mathematics, 73. Birkh\"auser Boston, Inc., Boston, MA, (1988). 339 pp.
\bibitem{ns}M. S. Narasimhan, C. S. Seshadri, Stable and unitary vector bundles on a compact Riemann surface, {\em Ann. of Math.} (2) {\bf 82} (1965) 540-567. 
\bibitem{prz}K. Polydorou, A. Roc\'en, M. J. Zabzine, 7D supersymmetric Yang-Mills on curved manifolds, {\em J. High Energ. Phys.} {\bf 152} (2017), pp 42.
\bibitem{rum}H. Rummler, Quelques notions simples en g\'eom\'etrie riemannienne et leurs applications aux feuilletages compacts, {\em Comment. Math. Helv.} {\bf 54} (1979), no. 2, 224-239. 
\bibitem{tian}G. Tian, Gauge theory and calibrated geometry. I, {\em Ann. of Math.} (2) {\bf 151} (2000), no. 1, 193-268.
\bibitem{uhl}K. Uhlenbeck, Connections with $L^{p}$ bounds on curvature, {\em Comm. Math. Phys.} {\bf 83} (1982), no. 1, 31-42. 
\bibitem{uy}K. Uhlenbeck, S.-T. Yau, On the existence of Hermitian-Yang-Mills connections in stable vector bundles, {\em Comm. Pure Appl. Math.} {\bf 39} (1986), no. S, suppl., S257-S293. 
\bibitem{weh}K. Wehrheim,  Uhlenbeck compactness. EMS Series of Lectures in Mathematics. European Mathematical Society (EMS), Z\"urich, 2004. viii+212 pp.
\bibitem{wlod}J. W\l odarczyk, Resolution of singularities of analytic spaces. Proceedings of G\"okova Geometry-Topology Conference 2008, 31-63, G\"okova Geometry/Topology Conference (GGT), G\"okova, 2009.
\end{thebibliography}

\end{document}